\documentclass[11pt]{article}
\usepackage{arxiv}

\usepackage{amsthm,amssymb}
\usepackage{mathtools,thmtools}
\usepackage{bm,esint}
\usepackage{color}
\usepackage{float,graphicx,subcaption}
\usepackage{cancel}
\usepackage[shortlabels]{enumitem}
\usepackage{mathrsfs}  
\usepackage{bbm}
\usepackage{euscript}
\usepackage{hyperref}
\usepackage{cleveref}
\usepackage{upgreek}

\usepackage[color=blue!20!white]{todonotes}

\newcommand{\mfK}{\mathsf{K}}

\usepackage{color}
\definecolor{darkred}{rgb}{.6,0,0}
\definecolor{darkblue}{rgb}{0,0,.7}
\definecolor{darkgreen}{rgb}{0,.7,0}
\definecolor{darkbrown}{rgb}{0.8,0.4,0.4}

\newcommand{\ams}[1]{{\color{black}{#1}}}
\newcommand{\sam}[1]{{\color{black}{#1}}}
\newcommand{\zongyi}[1]{{\color{black}{#1}}}

\newcommand{\per}{\mathrm{per}}

\renewcommand{\tilde}{\widetilde}
\renewcommand{\hat}{\widehat}

\newcommand{\embeds}{{\hookrightarrow}}

\renewcommand{\bar}{\overline}

\newcommand{\slot}{{\,\cdot\,}}

\newcommand{\T}{\mathbb{T}}
\newcommand{\R}{\mathbb{R}}
\newcommand{\C}{\mathbb{C}}
\newcommand{\E}{\mathbb{E}}
\newcommand{\N}{\mathbb{N}}

\newcommand{\Lip}{\mathrm{Lip}}

\newcommand{\cE}{\mathcal{E}}

\newcommand{\cK}{\mathcal{K}}
\newcommand{\cL}{\mathcal{L}}
\newcommand{\cM}{\mathcal{M}}

\newcommand{\cQ}{\mathcal{Q}}
\newcommand{\cR}{\mathcal{R}}

\newcommand{\cV}{\mathcal{V}}

\newcommand{\cX}{\mathcal{X}}
\newcommand{\cY}{\mathcal{Y}}

\newcommand{\tQ}{\tilde{\cQ}}

\newcommand{\define}{\textbf}

\renewcommand{\hat}{\widehat}

\newcommand{\set}[2]{{\left\{ #1 \,\middle|\, #2 \right\}}}

\declaretheoremstyle[
  headfont=\normalfont\bfseries\itshape,
  numbered=yes,
  bodyfont=\normalfont,
  spaceabove=1em plus 0.75em minus 0.25em,
  spacebelow=2em,
  qed={\small$\Diamond$},
]{deflt}

\declaretheorem[style=deflt,numberwithin=section]{theorem}
\declaretheorem[style=deflt,sibling=theorem]{lemma}
\declaretheorem[style=deflt,sibling=theorem]{proposition}
\declaretheorem[style=deflt,sibling=theorem]{corollary}
\declaretheorem[style=deflt,sibling=theorem]{remark}

\declaretheorem[style=deflt,sibling=theorem]{claim}

\numberwithin{equation}{section}

\newcommand{\dc}{{d_c}}

\title{\ams{Nonlocality and Nonlinearity Implies Universality in Operator Learning}}
\renewcommand{\shorttitle}{\ams{Nonlocality and Nonlinearity Implies Universality}}

\author{
Samuel Lanthaler
\thanks{Address: California Institute of Technology,
1200 E. California Blvd., MC 305-16,
Pasadena, CA, 91125, 
Email: \texttt{slanth@caltech.edu}, Phone: 626-395-2983} 
\\
California Institute of Technology
\And
Zongyi Li \\
California Institute of Technology
\And
Andrew M. Stuart \\
California Institute of Technology
}

\date{\today}

\ifpdf
\makeatletter
\hypersetup{
  pdftitle={\@title},
  pdfauthor={\@author}
}
\makeatother
\fi

\begin{document}
\maketitle

\begin{abstract}
\ams{Neural operator architectures approximate operators between infinite-dimensional Banach spaces of functions. They are gaining increased attention in computational science and engineering, due to their potential both to accelerate traditional numerical methods and to enable data-driven discovery. As the field is in its infancy
basic questions about minimal requirements for universal approximation remain open. It is clear that any
general approximation of operators between spaces of functions must be both nonlocal and nonlinear. In
this paper we describe how these two attributes may be combined in a simple way to deduce universal
approximation. In so doing we unify the analysis of a wide range of neural operator architectures
and open up consideration of new ones.

A popular variant of neural operators is the Fourier neural operator (FNO). Previous analysis proving universal operator approximation theorems for FNOs resorts to use of an unbounded number of Fourier modes, relying on intuition from traditional analysis of spectral methods. The present work challenges this point of view: (i) the work reduces FNO to its core essence, resulting in a minimal architecture termed the ``averaging neural operator'' (ANO); and (ii) analysis of the ANO shows that 
even this minimal ANO architecture benefits from universal approximation. This result is obtained based on only a spatial average as its only nonlocal ingredient (corresponding to retaining only a \emph{single} Fourier mode in the special case of the FNO, taking the analysis far from that of spectral methods.)
The analysis paves the way for a more systematic exploration of nonlocality, both through the development of new operator learning architectures and the analysis of existing and new architectures. 
Numerical results are presented which give insight into complexity issues
related to the roles of channel width (embedding dimension) and number of Fourier modes.}
\end{abstract}

\section{Introduction}

\subsection{Motivation and Literature Review}
\label{ssec:mlr}

Neural networks \cite{DLbook} are receiving growing interest in computational science and engineering. In many applications of interest in the sciences, the task at hand is to approximate an underlying operator, which defines a mapping between two infinite-dimensional Banach spaces of functions. New neural network-based frameworks, termed \emph{neural operators} \cite{anandkumar_neural_2020,bhattacharya_model_2021,lu_learning_2021,kovachki2021neural}, generalize neural networks to this infinite-dimensional setting, with the aim of learning such operators from data. They have the potential to \emph{accelerate} traditional numerical methods when a mathematical description, often in the form of a partial differential equation, is known. When no model is available, this data-driven methodology has the potential to \emph{discover} the underlying input-output map.

Several neural operator architectures have been proposed in recent years.
The first neural operator architecture (i) appeared in \cite{chen_universal_1995} and was subsequently
generalized by adding depth to the underlying networks (DeepONet) \cite{lu_learning_2021};
further extensions include \cite{mionet,seidman2022nomad,lanthaler2022nonlinear,patel2022variationally}. 
DeepONets have been successfully deployed in a variety of application areas \cite{di_leoni_deeponet_2021,mao_deepmandmnet_2021,cai_deepmmnet_2021,zhang2022belnet}. Concurrently with DeepONet
a number of other approaches to operator learning have also been developed and successfully deployed including: (ii) methods which combine ideas from principal component analysis with neural networks (PCA-Net) \cite{hesthaven_nonintrusive_2018,bhattacharya_model_2021}; (iii) operator learning based on random features, which may be viewed as Monte Carlo approximation of kernel methods, as introduced in \cite{nelsen2021random}; and (iv) the wide class of neural operators introduced in \cite{anandkumar_neural_2020,kovachki2021neural}. The class
of methods (iv) defines neural operators in analogy with conventional neural networks, but appends the weight matrices in the hidden layers with additional linear integral operators acting on the input function. Special cases of this framework include graph neural operators \cite{anandkumar_neural_2020,li_multipole_2020} and the Fourier neural operator (FNO) \cite{li_fourier_2021}. The FNO in particular has received considerable interest due to its state-of-the-art performance on many tasks. 
The FNO, and extensions thereof, remains a very active research direction: see, for example, \cite{wen2022u,li2022fourier,you2022learning,pathak2022fourcastnet,li2021physics,bonev2023spherical}. We also mention the closely related convolutional neural operator of \cite{raonic2023} and a related Fourier-based approach in  \cite{patel2018nonlinear,patel2021physics}.
The reliance on a Fourier basis limits the basic form of the FNO to periodic 
geometries although, in that setting, use of the fast Fourier transform (FFT)
allows for efficient computations with total number of Fourier components limited
only by the grid resolution.
There is a set of papers which seek to address the restriction to periodic geometry by replacing the
FFT with a wavelet \cite{tripura2023} or multi-wavelet \cite{gupta2021multiwavelet} 
transform; in this context, we also mention Laplace neural operators \cite{cao2023lno}, which extends the FNO by replacing the Fourier transform/inversion step with a Laplace
transform/inversion step. Another extension of FNO, so-called Neural Operator on Riemannian Manifolds (NORM) \cite{laplaceno}, generalize
the FNO to use arbitrary orthogonal eigenfunctions of the Laplace-Beltrami operator on
any given spatial domain. Another approach to extending the scope of the FFT is to learn transformations of arbitrary domains into unit cubes, suitable for 
application of the FFT \cite{li2022fourier}.
Most closely related to the approach proposed in the present paper is the low rank neural operator \cite[Sect. 4.2]{kovachki2021neural} which, however, has a more complicated architecture than that proposed here. 

There is a growing body of approximation theory for certain classes of operators arising from
PDEs, using Galerkin and Taylor approximation in high dimensional spaces;
see \cite{cohen2015approximation,kaarnioja2023lattice} for recent overviews of this literature,
in particular in the context of the Darcy flow problem as pioneered in the paper \cite{cohen2010convergence}.
However the empirical success of neural operators on very wide classes of operator learning problems,
beyond those covered by classical approximation theoretic approaches,
suggests the value of developing theory for neural operators.
A prerequisite for the success of any neural operator architecture, and its use in a diverse range of applications, is a \emph{universal approximation} property. Universal approximation results are known for DeepONets \cite{chen_universal_1995,lanthaler_error_2021}, PCA-Net \cite{bhattacharya_model_2021}, for general neural operators \cite{kovachki2021neural} and for the (more constrained) Fourier neural operator \cite{kovachki_universal_2021}. As most operators of interest are not only nonlinear, but also nonlocal, any successful operator learning framework must share both of these properties. While the need for nonlinearity is well-understood even for ordinary neural networks, nonlocality is a requirement that is more specific to the function-space setting; it is a crucial ingredient for operator learning. 
The headline message of this paper
is that universal approximation can be obtained in general geometries, with nonlocality introduced
using only a low-rank operator of fixed finite rank, and is not restricted to periodic domains.

Fourier neural operators introduce nonlocality via the addition of a nonlocal operator in each hidden layer, which acts on the Fourier modes of the input function by matrix multiplication. Available analysis of FNOs in \cite{kovachki_universal_2021} mostly rests on an analogy between FNOs and spectral methods: higher approximation accuracy is established by retaining larger numbers of Fourier modes. This interpretation is at odds with practical experience with FNOs.
FNOs are typically implemented with a first layer which lifts  the input,
a scalar or vector-valued function, to a vector-valued function where the vector
dimension (referred to as the number of channels, also named the model width) is much higher than that of the input function itself. In some circumstances it can be more beneficial to increase the number of channels rather than to retain more Fourier modes in the architecture. A first theoretical insight into this empirical observation has recently been achieved in \cite{lanthaler2022nonlinear} from the perspective of nonlinear reconstruction, where the analysis of specific settings revealed clear benefits of increasing the number of channels, and showed that, in those specific cases, arbitrary accuracy can be achieved when retaining only a \emph{fixed} number of Fourier modes; the same paper also
contains numerical experiments which exemplify the theoretical insight. These and similar results indicate that our theoretical understanding of neural operators, which has been largely guided by experience with traditional numerical methods, is still incomplete, and suggests a need for further analysis to improve our understanding of the precise role of nonlocality in operator learning. 

These observations motivate the present work, in which we seek a deeper understanding of ``how much'' nonlocality is needed in neural operator architectures. The surprising result of the present work is that even very simple nonlocality, in the form of an integral \emph{average}, is already sufficient for universality in operator learning. This is astonishing if one contrasts the simplicity of a humble average with the often complex nonlocal PDE solution operators of interest in engineering and the sciences, which, for example, give rise to turbulence and other complex phenomena. The present work thus paves the way for a more systematic exploration of nonlocality in operator learning. 

\ams{We highlight the fact that theory quantifying the computational complexity of achieving a given error
appears very difficult for neural operators. It is currently limited to random features models \cite{lanthaler2024error}, to linear problems \cite{de2023convergence} and to specific problems such as those with holomorphic structure and deploying specific network constructions \cite{marcati2023exponential,Herrmann2024Operator}; numerical studies of the cost-accuracy trade-off beyond these special settings may be found in \cite{de2022cost}. 
Thus, although they do not quantify complexity, universal approximation theorems
have an important role to play as neural operator methodology is developed: they highlight the importance of
certain high level design choices in neural operators.
We also note that our analysis and numerics provide an improved understanding of the trade-offs between
increasing the channel width versus increasing the number of Fourier modes, in the FNO architecture. Indeed our numerical experiments point to problem-dependent subtleties in optimizing this trade-off.}

In the next subsection, we \sam{first provide an overview of the general structure of neural operators in a unified framework, which encompasses recently proposed architectures \cite{li2022fourier,bonev2023spherical}. We then proceed to boil this general structure down to its bare essence,} the averaging neural operator (ANO). \sam{The ANO builds on only two minimal ingredients, nonlinearity by composition with shallow neural networks, and nonlocality via a spatial average.} We state, and sketch the
proof of, two universal approximation theorems for the ANO; proof details are given in the Appendix \ref{app:pf-univ}.
The ANO is a subclass of numerous neural operators, including many of those listed
in the preceding literature review; we thus deduce, in section \ref{sec:connection},
many universal approximation theorems from properties of the ANO. We follow this, in section \ref{sec:N}, with numerical
experiments illustrating some of the implications of our work. Section \ref{sec:D}
contains concluding discussions. 

\subsection{Neural Operator} 
\label{ssec:NNO}

\sam{
We now provide an overview of neural operator architectures which are of particular relevance to this work. To differentiate this generic class from specific instances, such as those further discussed in Section \ref{sec:connection}, we will refer to this general class as \emph{nonlocal neural operators} (NNO).
}
Let $\Omega$ denote a bounded domain in $\R^d$ (or potentially a manifold) and let $\cX(\Omega; \R^o), \cY(\Omega; \R^o)$ and $\cV(\Omega; \R^o)$ denote Banach spaces of $\R^o-$valued functions over
$\Omega.$
The NNO is defined as a mapping $\Psi: \cX(\Omega; \R^k) \to \cY(\Omega; \R^{k'})$ which can be written as a composition of the form $\Psi = \cQ \circ \cL_L \circ \dots \circ \cL_1 \circ \cR$, consisting of a lifting layer $\cR$, hidden layers $\cL_\ell$, $\ell=1,\dots, L$, and a projection layer $\cQ$. Given a \define{channel dimension} $\dc$, the \define{lifting layer} $\cR$ is given by a mapping
\begin{align}
\label{eq:R}
\cR: \cX(\Omega; \R^k) \to \cV(\Omega; \R^\dc), \quad u(x) \mapsto R(u(x),x),
\end{align}
where $R: \R^k \times \Omega \to \R^\dc$ is a learnable neural network acting
between finite dimensional Euclidean spaces. 
For $\ell = 1, \dots, L$ (the number of \define{hidden layers}) 
and for $m=0,\dots, M$ (the number of \define{modes}) 
choose functions $\psi_{\ell,m}, \phi_{\ell,m}: \Omega \to \R^{d_c}$.
For $\ell = 1, \dots, L$, each \define{hidden layer} $\cL_\ell$ is of the form
\begin{align}
\label{eq:gexp}
(\cL_{\ell} v)(x) := 
\sigma\left(
W_\ell v(x) + b_\ell + \sum_{m=0}^M \langle T_{\ell,m} v, \psi_{\ell,m} \rangle_{L^2(\Omega;\R^{d_c})} \phi_{\ell,m}(x)
\right).
\end{align}
Each hidden layer defines a mapping $\cL_\ell: \cV(\Omega; \R^{d_c}) \to \cV(\Omega; \R^{d_c}).$ For $\ell = 1, \dots, L$ and for $m=0,\dots, M$ 
the matrices $W_\ell, T_{\ell,m}\in \R^{\dc\times \dc}$, 
and bias $b_{\ell}\in \R^\dc$ are learnable parameters. 
The activation function $\sigma: \R \to \R$ acts as a Nemitskii-operator, component-wise on inputs; i.e. for a vector-valued function $v(x) = (v_1(x), \dots, v_\dc(x))$, we define $\sigma(v(x)) := (\sigma(v_1(x)), \dots, \sigma(v_\dc(x)))$ for $x\in \Omega$. 
Throughout the paper the activation function $\sigma: \R \to \R$ is assumed to be smooth, $\sigma \in C^\infty(\R)$, nonpolynomial and Lipschitz continuous.
Finally, the \define{projection layer} $\cQ$ is given by a mapping,
\begin{align}
\label{eq:Q}
\cQ: \cV(\Omega; \R^\dc) \to \cY(\Omega; \R^{k'}), \quad v(x) \mapsto Q(v(x),x),
\end{align}
where $Q: \R^{\dc}\times \Omega \to \R^{k'}$ is
also a learnable neural network acting
between finite dimensional Euclidean spaces. \ams{Note that the form of the lifting and projection
layers allows for \emph{positional encoding.}}

\sam{
The general framework summarized above reduces to the FNO in a periodic geometry and if the expansion functions are chosen as Fourier basis functions, indexed by $m$ and independent of $\ell$. It also relates to other architectures, such as wavelet operators \cite{gupta2021multiwavelet,tripura2023}, or the Laplace eigenbasis neural operators of \cite{laplaceno}. The central question to be addressed in the following is this: ``which minimal assumptions have to be imposed on the expansion functions $\psi_{\ell,m}$ and $\phi_{\ell,m}$, to ensure universal approximation of the resulting architecture?''
}

\section{Averaging Suffices for Universal Approximation}
\label{sec:A}

In this section we prove that the use of a simple average provides 
enough nonlocality, within the broader NNO structure, leading to the two
universal approximation Theorems \ref{thm:universal10}
and \ref{thm:universal20} stated below. \sam{The architecture used to prove
these theorems is a subclass of many concrete instantiations of the general NNO architecture,
and hence implies universality approximation results for NNO as corollaries (cf. Section \ref{sec:connection}, below).}
When the domain is periodic,
the resulting averaging neural operator is also a special case of the FNO, when only the zeroth Fourier mode is retained. Thus, our universality result 
also implies the universality of FNOs, even when only computing  with a single,
constant, Fourier mode. Thus, the theoretical results in this section
provide new insight into the empirical observation, for example in the experiments in \cite{lanthaler2022nonlinear}, that increasing the number of channels is often more beneficial than increasing the number of Fourier modes in the practical training of FNOs. In subsection \ref{ssec:nonlocal} we discuss the nonlinearity of
neural networks between finite dimensional Euclidean spaces and emphasize
the role of nonlocality in operator learning. Subsection \ref{ssec:subclass}
defines the subclass of NNOs used for our analysis of universal approximation,
the topic of subsection \ref{ssec:UA}. Intuition for the
proof is given in subsection \ref{ssec:intuition}, and a proof sketch is provided in subsection
\ref{ssec:sketch}, with details left for the appendix.

\subsection{Nonlinearity and Nonlocality}
\label{ssec:nonlocal}

To achieve a universal approximation property, a neural operator architecture must necessarily define a nonlinear operator. 
In the context of ordinary neural networks acting as approximators of functions
between finite dimensional Euclidean spaces, it is well-known that (nonpolynomial) nonlinearity essentially also represents a \emph{sufficient} condition for their universal approximation property \cite{pinkus1999approximation,BAR1,HOR1,Cy1}. This should be contrasted with neural operators, where nonlinearity alone is not sufficient to ensure universality; to illustrate this latter fact, we note that even a single layer of an ordinary neural network actually gives rise to a nonlinear operator, which maps an input function to an output function by composition, 
\begin{align}
\label{eq:nonlin}
u(x) \mapsto \sigma( W u(x) + b).
\end{align} 
Here, $W$ and $b$ represent the weights and biases, respectively, and $\sigma$ denotes the activation function.
Despite being highly nonlinear, operators of such a compositional form cannot be universal. The main reason for this is that any such $\Psi$ is \emph{local}, in the sense that the value $\Psi(u)(x)$ of the output function at a given evaluation point $x$ depends only on the value of the input function $u(x)$ at that same evaluation point. As a consequence, such mappings are not able to approximate even simple operators $\Psi^\dagger$ with a \emph{nonlocal} dependence on the input, such as the shift operator $\Psi^\dagger(u)(x) := u(x+h)$ for fixed $h\ne 0$.

The arguably simplest example of a nonlocal operator with dependence on \emph{all} point values $u(x)$, $x\in \Omega$, is given by averaging the input function $u$ over its domain $\Omega$:
\begin{align}
\label{eq:avg}
u(x) \mapsto \fint_{\Omega} u(y) \, dy:= \frac{1}{|\Omega|} \int  u(y) \, dy.
\end{align}
Clearly, such averaging represents only a very special case of a nonlocal operator. In general, nonlocal operators can have a much more complicated dependence on the function values $\set{u(x)}{x\in \Omega}$ over the  whole domain, as well as the mutual correlations of these values.

\subsection{Averaging Neural Operator: a Special Subclass of the NNO}
\label{ssec:subclass}
In this section, we define a special subclass of the NNO, which combines nonlinearity by composition \eqref{eq:nonlin} with nonlocality by averaging \eqref{eq:avg}.
Recall that we have restricted this paper to consideration of only smooth, nonpolynomial
and Lipschitz continuous activation functions $\sigma$. 
Generalization of our results to other activation functions is possible, but would require additional technical assumptions in our main results.
We now define the subclass of NNOs found by simplifying to a special subclass of
hidden layers of the form,
\begin{align}
\label{eq:hidden}
\cL: \cV(\Omega; \R^\dc) \to \cV(\Omega; \R^\dc), \quad \cL(v)(x) := \sigma\left(W v(x) + b + \fint_{\Omega} v(y) \, dy\right).
\end{align}
\sam{
More specifically, we will focus on the case of \emph{a single hidden layer}. This results in the following \emph{averaging neural operator} (ANO) architecture:
\[
\Psi: \cX(\Omega; \R^o) \to \cY(\Omega;\R^o), \quad \Psi(u) = \cQ \circ \cL \circ \cR(u),
\]
where the lifting and projection mappings, $\cR$ and $\cQ$, are given by \eqref{eq:R} and \eqref{eq:Q}, respectively, with $R$ and $Q$ single-hidden layer neural networks of width $\dc$, and where $\cL$ is of the form \eqref{eq:hidden}.

We note that this ANO is a specific subclass of the NNO. Due to its minimal structure, the ANO depends on only one \emph{hyperparameter}; the lifting dimension $\dc$. The \emph{tunable parameters} of ANO are represented by the weight matrix $W \in \R^{\dc\times \dc}$ and bias $b \in \R^{\dc}$ in the hidden layer $\cL$, and the internal weights and biases of the ordinary neural networks $R$ and $Q$, which define the lifting and projection layers, respectively.
}

\subsection{Universal Approximation}
\label{ssec:UA}

Despite the apparent simplicity of averaging as the only nonlocal ingredient in the definition of the averaging neural operator, the following theorem shows that this architecture is universal:

\begin{theorem}
\label{thm:universal10}
Let $\Omega \subset \R^d$ be a bounded domain with Lipschitz boundary. For given integers $s,s'\ge 0$, let $\Psi^\dagger: C^{s}(\bar{\Omega}; \R^k) \to C^{s'}(\bar{\Omega}; \R^{k'})$ be a continuous operator, and fix a compact set $\mfK\subset C^s(\bar{\Omega};\R^k)$. Then for any $\epsilon > 0$, there exists an averaging neural operator $\Psi: \mfK \subset C^s(\bar{\Omega}; \R^k) \to C^{s'}(\bar{\Omega}; \R^{k'})$ such that 
\[
\sup_{u\in \mfK} \Vert \Psi^\dagger(u) - \Psi(u) \Vert_{C^{s'}} \le \epsilon.
\]
\end{theorem}

The previous result is formulated for operators on spaces of continuously differentiable functions. However, it is possible to obtain similar results in many other natural settings. To illustrate the generality of the underlying ideas, we provide a corresponding result in the scale of Sobolev spaces $W^{s,p}$:
\begin{theorem}
\label{thm:universal20}
Let $\Omega \subset \R^d$ be a bounded domain with Lipschitz boundary. 
For given integers $s,s'\ge 0$, and reals $p,p'\in [1,\infty)$, let $\Psi^\dagger: W^{s,p}(\Omega; \R^k) \to W^{s',p'}(\Omega; \R^{k'})$ be a continuous operator. Fix a compact set $\mfK\subset W^{s,p}(\Omega;\R^k)$ of bounded functions, $\sup_{u\in K} \Vert u \Vert_{L^\infty} <\infty$. Then for any $\epsilon > 0$, there exists an averaging neural operator $\Psi: W^{s,p}(\Omega; \R^k) \to W^{s',p'}(\Omega; \R^{k'})$ such that 
\[
\sup_{u\in K} \Vert \Psi^\dagger(u) - \Psi(u) \Vert_{W^{s',p'}} \le \epsilon.
\]
\end{theorem}

\sam{
A straight-forward consequence of the above results is that any more general NNO architecture, with hidden layers of the form \eqref{eq:gexp}, is universal as long as an average can be represented; under this condition, the resulting NNO reduces to ANO with a specific setting of the tunable weights (by setting certain parameters to zero). This is ensured, provided that there exists an index $m$ such that $\psi_{\ell,m}, \phi_{\ell,m} \equiv \text{const.}$ are constant functions, and satisfied for most architectures that are successfully employed in practice (cf. Section \ref{sec:connection} for several examples). 

\begin{remark}
Theorems \ref{thm:universal10} and \ref{thm:universal20} show that the ANO architecture is universal in approximating a large class of operators, uniformly over a compact set of input functions $K$. In practice, neural operators are often trained by minimizing an empirical loss,
\[
\cL(\Psi) = 
\frac1N \sum_{n=1}^N \Vert \Psi^\dagger(u_n) - \Psi(u_n)\Vert^2_{W^{s',p'}(\Omega)}, 
\quad u_1,\dots, u_N \sim \mu,
\]
where the data are iid random samples from an underlying input probability measure $\mu$. A popular prototypical choice is sampling input functions from a Gaussian random field. In this case, the set of input functions is no longer bounded, and hence not compact. As previously observed in \cite[Remark 17 and Thm. 18]{kovachki_universal_2021}, we note that it is possible to combine universality results with respect to the uniform norm over a compact set $K$, with a technical cut-off argument, to derive corresponding results in expectation; i.e. for given $\epsilon>0$, it is possible to show the existence of $\Psi$, such that 
\[
\E_{u\sim \mu}\left[ \Vert \Psi(u) - \Psi^\dagger(u) \Vert^2_{W^{s',p'}}\right]
< \epsilon,
\]
including the case where $\mu$ does not have compact support.
\end{remark}

}

\subsection{Intuition}
\label{ssec:intuition}

\sam{
Before providing a sketch of the main technical elements that go into the proof of Theorems \ref{thm:universal10} and \ref{thm:universal20}, in the next subsection \ref{ssec:sketch}, we would like to provide some further remarks on the general intuition, as well as possible variants and extensions of these theorems.

\paragraph{Encoder-Decoder Structure.}
The main theoretical insight of this work, brought out in the ANO architecture, is the identification of a hidden encoder-decoder structure which is inherent in NNOs. This structure becomes particularly apparent when setting the matrix $W$ and bias $b$ in the hidden layer \eqref{eq:hidden} zero, in which case we note that the mapping $\cL\circ \cR: \cX \to \R^\dc$, $u \mapsto \cL\circ \cR(u)$ can be thought of as a nonlinear encoding of the input function by a (constant) vector $v\in \R^\dc$, while the mapping $\cQ: \R^\dc \to \cY$, $v\mapsto Q(v,\slot)$ is a nonlinear decoding of the corresponding output function $\Psi^\dagger(u)$. Indeed, with this choice of the hidden layer, the composition $\cL\circ \cR$ has constant output,
\[
v := \cL \circ \cR(u) = \sigma\left( \fint_{\Omega} R(u(x),x) \, dx \right) \in \R^{d_c},
\]
thus serving as an encoder, $\cL \circ \cR: \cX(\Omega;\R^o) \to \R^{d_c}$. The encoded input $v$ is in turn decoded by the neural network $\cQ: v \mapsto Q(v,\slot)$ to produce the output function $\Psi(u)(x) = Q(v,x)$. 

This encoder/decoder perspective provides further insight as to why the universal approximation property can hold even when relying on simple averaging as the only nonlocal ingredient; \ams{the main difficulty in proving such
a result lies in showing that composition of the input function with a non-linear transformation, and subsequent averaging, can encode the relevant features about the input.}

\paragraph{The Role of Positional Encodings.}

For the encoder-decoder point of view presented above, the explicit $x$-dependence (``positional encodings'') of the lifting and the projection layers appears crucial at first sight, and we would like to further clarify the precise sense in which this is so. Specifically, we would like to address the fact that in practical implementations, of e.g. FNO, it is standard to add such positional encoding to the inputs, while it is not common to add such features to the output layer. 

First, we point out that some form of positional encoding is necessary even from the point of view of universality. E.g. in the case of the FNO (on a periodic domain), positional encoding represents the only mechanism by which the architecture can break translation equivariance, i.e. ensure that the FNO $\Psi$ does not commute with the shift operator, $\tau_h \circ \Psi = \Psi \circ \tau_h$, where $\tau_h(u) = u(\slot +h)$. In many cases, such as for PDEs with non-constant coefficients, the operator of interest is not translation equivariant, making this mechanism for breaking translation equivariance crucial. In previous theoretical work on universality for FNO \cite{kovachki_universal_2021}, translation equivariance was broken by assuming an architecture with $x$-dependent \emph{bias functions}, $b=b(x)$. However, in practical implementations, it is more common to add a positional encoding explicitly in the input layer $\cR: u \mapsto R(u(x),x)$, consistent with the formulation in the present work.

Second, we note that the explicit $x$-dependence in the output layer is not necessary, provided that pointwise matrix multiplication is included in the hidden layers: 
\[
v(x) \mapsto \sigma\left( W v(x) + b  + \fint_{\Omega} v(y) \, dy\right).
\]
Indeed, these pointwise operations provide the network a mechanism (akin to skip connections) to forward positional encoding information from the input layer to the output layer. Introducing an explicit $x$-dependence directly in the output layer $\cQ$ is thus not crucial from a universality point of view. We nevertheless assume such dependence, because (i) it allows us to simplify the proofs by avoiding technical details that would be required to make the above statement precise, and (ii) the assumed explicit $x$-dependence also points to a generalization of the architecture to approximate operators on input-/output-functions with distinct domains, as will be explained below.

\paragraph{Other Extensions and Variants.} There are several possible extensions and variants of our main theorems, two of which we briefly mention below.

Firstly, our derivation actually implies that neural operators of the form \eqref{eq:avg-no} are universal in approximating continuous operators $\Psi^\dagger: \cX(\Omega;\R^k) \to \cY(\Omega';\R^{k'})$ on \emph{distinct} domains $\Omega\subset \R^d$ and $\Omega'\subset \R^{d'}$, and with potentially different dimensions $d\ne d'$. We only need to replace $Q$ in \eqref{eq:avg-no} by a neural network mapping $Q': \R^\dc \times \Omega' \to \R^{k'}$. The input and output spaces $\cX, \cY$ can be any combination of $C^s(\bar{\Omega};\R^k), W^{s,p}(\Omega;\R^k)$ and $C^{s'}(\bar{\Omega}';\R^{k'}), W^{s',p'}(\Omega';\R^{k'})$, respectively, for any integers $s,s'\ge 0$ and reals $p,p' \in [1,\infty)$. Extension to this setting specifically requires that the pointwise matrix multiplication, $v(x) \mapsto Wv(x)$, in the hidden layer is removed, by setting $W=0$. It would be interesting to investigate the practical utility of this observation in numerical experiments, in the future.

Secondly, it should be straight-forward, although technically more involved, to extend the results of this work to operators mapping between functions defined on compact manifolds $M \subset \R^d$ in place of $\Omega$; this is particularly relevant to applications of neural operators to numerical weather forecasting \cite{pathak2022fourcastnet,bonev2023spherical}.
}

\subsection{Sketch Of The Proof Of Universality}
\label{ssec:sketch}

In the present section, we provide an overview of the proof of Theorems \ref{thm:universal10} and \ref{thm:universal20}. The details are included in Appendix \ref{app:pf-univ}, see Sections \ref{app:universal1} and \ref{app:universal2}, in particular. In the following overview, let $\cX = \cX(\Omega;\R^k)$ and $\cY = \cY(\Omega;\R^{k'})$ denote  spaces of continuously differentiable functions
(respectively Sobolev spaces) 
as in the statement of Theorem \ref{thm:universal10} 
(respectively Theorem \ref{thm:universal20}).

The first step in the proof of universality is to determine a reduction to a special class of operators, explained in detail in Section \ref{app:dense}, which
have the structure of an encoder-decoder. Starting from a general continuous operator $\Psi^\dagger: \cX \to \cY$, and given a compact set $\mfK \subset \cX$, we first show that for any $\epsilon > 0$, there exist $J\in \N$, elements $\eta_1,\dots, \eta_J \in \cY$ and continuous functionals $\alpha_1,\dots, \alpha_J: \cX \to \R$, such that 
\begin{equation}
\label{eq:psisimple0}
\sup_{u\in \mfK}
\left\Vert
\Psi^\dagger(u) - \sum_{j=1}^J \alpha_j(u) \eta_j
\right\Vert_{\cY}
\le 
\epsilon.
\end{equation}
Furthermore, our construction ensures that each of these functionals $\alpha_j$ actually defines a continuous functional on the larger input function space consisting of $L^1$-functions, i.e. we have continuous mappings $\alpha_j: L^1(\Omega;\R^k) \to \R$. 
Since any input function space $\cX$ considered in Theorems \ref{thm:universal10} and \ref{thm:universal20} embeds into $L^1(\Omega;\R^k)$, 
this first step in the proof reduces the original problem 
to the approximation of a continuous operator $\tilde{\Psi}^\dagger: L^1(\Omega;\R^k) \to \cY$, of the form,
\begin{align}
\label{eq:psisimple}
\tilde{\Psi}^\dagger(u) = \sum_{j=1}^J \alpha_j(u) \eta_j,
\end{align}
over a compact subset $\mfK \subset L^1(\Omega;\R^k)$ of uniformly bounded functions, $\sup_{u\in K} \Vert u \Vert_{L^\infty}<\infty$.  The structure of $\tilde{\Psi}$ is that of an encoder-decoder, with
the functionals $\alpha_j$ providing the encoding and the functions $\eta_j$ the
decoding. We refer to Proposition \ref{prop:simpler} for further details in the setting of continuously differentiable spaces, and Proposition \ref{prop:simpler2} for a corresponding statement in the scale of Sobolev spaces.

The second step in the proof, described in detail in Section \ref{app:alpha-approx}, concerns approximation of the continuous nonlinear functionals $\alpha_j: L^1(\Omega;\R^k) \to \R$, the encoders. Specifically, we show that they can be approximated to any desired accuracy $\epsilon$ by an averaging neural operator $\tilde{\alpha}: L^1(\Omega;\R^k) \to \R^J$ (with constant output functions), and with components $\tilde\alpha(u) = (\tilde\alpha_1(u),\dots, \tilde\alpha_J(u))$. To
be concrete,
\[
\sup_{u\in \mfK} \left|\alpha_j(u) - \tilde\alpha_j(u) \right| \le \epsilon,
\quad \forall \, j=1,\dots, J.
\]
This result is contained in Lemma \ref{lem:alpha-approx} of Section \ref{app:alpha-approx}, which constitutes the main ingredient in the proof of universality. 
More precisely, the constructed averaging neural operator $\tilde\alpha$ is of the form 
\[
\tilde\alpha: u \mapsto R(u(x),x) \mapsto v:= \sigma \left(\fint_{\Omega} R(u(y),y) \, dy \right) \mapsto q_1(v),
\]
where $R: \R^k \times \Omega \to \R^\dc$ and $q_1: \R^\dc \to \R^J$ are ordinary neural networks; see Remark \ref{rem:alpha-approx} after the proof of Lemma \ref{lem:alpha-approx}. This step builds on well-known approximation results for ordinary neural networks, summarized in Section \ref{app:nn-univ}, and combines them with an averaging operation.

The final third step in the proof of Theorem \ref{thm:universal10} is given in Section \ref{app:universal1} (see p. \pageref{pf:universal1}) and is built on
approximations of the functions $\eta_j$ which define the decoder. We first show that there exists another (ordinary) neural network $q_2: \R^J \times \Omega \to \R^{k'}$, such that 
$q_2(v,x) \approx \sum_{j=1}^J v_j \eta_j(x)$ for all relevant $(v,x)$,
approximating the decoding.
We then define a neural network $Q$ by the composition $Q(v,x):= q_2(q_1(v),x)$. 

Given the neural network $R$ and averaging layer $\cL(w) := \sigma\left(\fint_\Omega w(y) \, dy\right)$ from step two of the proof, and the neural network $Q$ from step three,
we show that  the averaging neural operator $\Psi: L^1(\Omega;\R^{k}) \to \cY$, given by the composition,
\begin{align}
\label{eq:avg-no}
\Psi:
u(x)
\overset{\cR}\mapsto 
R(u(x),x) 
\overset{\cL}\mapsto 
v:= \sigma \left(\fint_{\Omega} R(u(y),y) \, dy \right)
\overset{\cQ}\mapsto Q(v,x),
\end{align}
approximates the encoder-decoder \eqref{eq:psisimple}. Concretely, given 
any $\epsilon>0$, we may design the neural networks $R$ and $Q$ so that
\[
\sup_{u\in \mfK} 
\left\Vert 
\Psi(u) - \sum_{j=1}^J \alpha_j(u) \eta_j
\right\Vert_{\cY} \le \epsilon.
\]
Step one of the proof shows that an encoder-decoder
can be constructed to $\epsilon-$approximate the true map $\Psi^\dagger$, as stated in \eqref{eq:psisimple0}.
 It thus follows that averaging neural operators are universal approximators of general continuous operators $\Psi^\dagger: \cX \to \cY$.

Two important properties of $\cX$ and $\cY$ that are used in this proof are that (a) $\cX$ has a continuous embedding in $L^1(\Omega;\R^k)$, and (b) that functions $\eta_1,\dots, \eta_J\in \cY(\Omega;\R^{k'})$ can be approximated by ordinary neural networks, i.e. for any $\epsilon > 0$, there exist neural networks $\tilde{\eta}_1,\dots, \tilde{\eta}_J: \Omega \to \R^{k'}$, such that $\Vert \eta_j - \tilde{\eta}_j \Vert_{\cY} \le \epsilon$. In addition, a more technical ingredient for the reduction to an operator of the form \eqref{eq:psisimple} is the existence of a ``mollification'', defining a family of maps $\cM_\delta: L^1(\Omega;\R^k) \to \cX$, and satisfying $\Vert \cM_\delta u - u \Vert_{\cX} \to 0$ as $\delta \to 0$, for any $u\in \cX\subset L^1$. Relevant mathematical background, including a result on boundary-adapted mollification, is summarized in Appendix \ref{app:moll}. For simplicity of exposition, rather than striving for the greatest generality, we have formulated our universal approximation theorem for two concrete families of function spaces, $C^s$ and $W^{s,p}$, which satisfy the required properties and which cover most settings encountered in applications.

\begin{remark}
    We note that additional properties, such as boundary conditions, are often implicitly enforced, provided that the operator is approximated with respect to a suitable norm on the output function space. For example, if the output functions belong to the Sobolev space $W^{1,2}_0(\Omega;\R^{k'})$, respecting
    a homogeneous Dirichlet boundary condition, then approximation of the operator with respect to the $W^{1,2}(\Omega;\R^{k'})$-norm automatically ensures that the boundary conditions are at least approximately satisfied, as a consequence of the trace theorem \cite[Sect. 5.5]{evans2022partial}.
\end{remark}

In the outline above, the parameter $J$ directly relates to the required width of the network $Q$ defining the projection layer $\cQ$. Thus, $J$ can be interpreted as a measure of the complexity of the output function space. The channel dimension $\dc$ corresponds to the number of necessary ``features'', encoded via the components $\sigma \left( \fint_{\Omega} R_1(u(y),y) \, dy\right), \dots, \sigma\left(\fint_{\Omega} R_\dc(u(y),y) \, dy\right)$ of the composition $\cL \circ \cR$, needed to approximate the functionals $\alpha_j$ to a given accuracy. These nonlinear features provide an encoding of the input function, and $\dc$ is a measure of the information 
content that needs to be retained.

\section{Connection With Other Neural Operator Architectures}
\label{sec:connection}

In this section we discuss the implications of the work in the
previous section for,  and links with,
other neural operator architectures.
Several neural operator architectures can be viewed as generalizations of the ANO, 
and a specific choice of the weights in the hidden layers will reproduce the simple averaging considered in the present work. Examples include the Fourier neural operator, the wavelet neural operator and neural operator with general integral kernel.
As a consequence, the universal approximation property for the ANO immediately implies corresponding universal approximation results for these architectures. Furthermore, there are links to emerging and existing
neural operator architectures such as the Laplace neural operator \cite{laplaceno}, DeepONet \cite{lu_learning_2021} and NOMAD \cite{seidman2022nomad}.
We now detail these implications and links.

\subsection{Neural Operator With General Integral Kernel}
Recall that the hidden layers of an ANO are of the form 
\begin{align}
\label{eq:hidden1}
\cL_\ell(v)(x) = 
\sigma\left(
W_\ell v(x) + b_\ell + (\cK_\ell v)(x)
\right),
\end{align}
where $\cK_\ell$ is the nonlocal operator given by averaging, $\cK_\ell v = \fint_{\Omega} v(y) \, dy$. This is a special case of the general form of neural operators defined in \cite{kovachki2021neural} (General Neural Operator), where the hidden layers are of the form \eqref{eq:hidden1}, with nonlocal operator given by integration against a matrix-valued integral kernel $K_\ell(x,y) \in \R^{\dc \times \dc}$:
\begin{align}
\label{eq:kernel}
(\cK_\ell v)(x) = \int_{\Omega} K_\ell(x,y) v(y) \, dy.
\end{align}
The particular choice $K_\ell(x,y) \equiv |\Omega|^{-1} I_{\dc\times \dc}$ with $I_{\dc\times \dc}$ the identity matrix, recovers the ANO. Hence, as an immediate consequence of the universal approximation theorems \ref{thm:universal10} and \ref{thm:universal20}, we obtain:

\begin{corollary}[General Neural Operator]
\label{cor:1}
The neural operator architecture with general integral kernel is universal in the settings of Theorem \ref{thm:universal10} and Theorem \ref{thm:universal20}.
\end{corollary}
This result generalizes the universal approximation theorem of \cite{kovachki2021neural} to even more general input and output spaces.

\subsection{Low-Rank Neural Operator}

A particular choice of the general integral kernel \eqref{eq:kernel}, proposed in  \cite[Sect. 4.2]{kovachki2021neural}, leads to the low-rank neural operator. Here, the integral kernel $K_\ell(x,y)$ of each layer is expanded in the low-rank form,
\[
K_\ell(x,y) = \sum_{k=1}^r \phi_{\ell,j}(x) \psi_{\ell,j}(y)^T,
\]
where $\phi_{\ell,j}, \psi_{\ell,j}: \Omega \to \R^\dc$ are vector-valued functions determined by neural networks, and optimized during the training of the low-rank neural operator. If $r=\dc$, and choosing the neural networks to be constant, $\phi_{\ell,j} \equiv |\Omega|^{-1} e_j$, $\psi_{\ell,j} \equiv e_j$ with $e_j \in \R^\dc$ the $j$-th unit vector, we obtain $K_\ell(x,y) = |\Omega|^{-1} I_{\dc\times \dc}$. This recovers the ANO. Thus, we conclude:
\begin{corollary}[Low-Rank Neural Operator]
\label{cor:2}
The neural operator architecture with low-rank kernel with $r\ge \dc$ is universal in the settings of Theorem \ref{thm:universal10} and Theorem \ref{thm:universal20}.
\end{corollary}

Note that the NNO proposed in the present work differs from what is proposed in \cite[Sect. 4.2]{kovachki2021neural}: here we pre-specify the functions $\phi_{\ell,j}, \psi_{\ell,j}$, and learn only a multiplier; in contrast, the pre-existing work proposed learning these
functions as neural networks. In this sense, the NNO can be viewed as a special case
of the low-rank neural operator.

\subsection{Fourier Neural Operator}
The Fourier neural operator is obtained by restricting to a class of convolutional integral kernels in \eqref{eq:kernel}, i.e. kernels satisfying $K_\ell(x,y) = K_\ell(x-y)$, where $K_\ell(x)$ is a trigonometric polynomial of the form $K_\ell(x) = \sum_{|k|\le k_{\mathrm{max}}} \hat{P}_{\ell,k} e^{ikx}$ with $\hat{P}_{\ell,k}\in \C^{\dc \times \dc}$. Upon setting $\ell=1$, $k_{\mathrm{max}} =0$, and $\hat{P}_{1,0} = |\Omega|^{-1} I_{\dc\times \dc}$, we again recover the averaging neural operator. Thus, Theorem \ref{thm:universal10} immediately implies:
\begin{corollary}[Fourier Neural Operator]
\label{cor:3}
The Fourier neural operator architecture is universal in the settings of Theorem \ref{thm:universal10} and Theorem \ref{thm:universal20}.
\end{corollary}
This result generalizes the previous universal approximation theorem of \cite{kovachki_universal_2021} to arbitrary Sobolev input and output spaces.

\subsection{Wavelet Neural Operator}
Multi-wavelet neural operators have been introduced in \cite{gupta2021multiwavelet};
see also \cite{tripura2023}. Following the approach in \cite{gupta2021multiwavelet}, the integral kernel $K_\ell(x,y)$ in a given hidden layer is expanded in terms of a wavelet basis $\psi_0,\psi_1,\dots, \psi_N$,
\[
K_\ell(x,y) = \sum_{i,j=0}^N C^{(i,j)} \psi_i(x) \psi_j(y).
\]
\sam{Here, the coefficients $C^{(i,j)} \in \R^{\dc\times \dc}$ are themselves matrix-valued.}
We consider the case of a single hidden layer, $\ell=1$, and the setting where the span of the wavelet basis $\psi_0,\dots$ includes constant functions, and hence we assume that $\psi_0 \equiv \mathrm{const.}$ is constant. The ANO is then again a special case, and we have:
\begin{corollary}[Wavelet Neural Operator]
\label{cor:4}
If the wavelet basis includes a constant function $\psi_0(x) \equiv \mathrm{const.}$, then the wavelet neural operator architecture is universal in the settings of Theorem \ref{thm:universal10} and Theorem \ref{thm:universal20}.
\end{corollary}
This result provides a first universal approximation theorem for wavelet neural operators, and thus gives a firm theoretical underpinning for the methodology proposed in \cite{gupta2021multiwavelet}.

\subsection{Laplace Neural Operator}

Recently another alternative to the Fourier neural operator has been proposed in \cite{laplaceno}; this so-called \emph{Laplace neural operator}, is applicable to general bounded domains $\Omega \subset \R^d$, and indeed extends to functions defined on quite general smooth manifolds. The approach of \cite{laplaceno} employs an expansion in a (truncated) eigenbasis  of the Laplacian, $\phi_0, \phi_1, \dots, \phi_{K} \in L^2(\Omega)$, and defines the nonlocal kernel $\cK_\ell$ in the $\ell$-th hidden layer \eqref{eq:hidden1} by an expression of the form,
\[
\cK_\ell v := \sum_{k=0}^K \langle T_{\ell,k} v,\phi_k \rangle_{L^2} \phi_k,
\]
where $T_{\ell,k} \in \R^{\dc\times \dc}$ is a learnable matrix multiplier for $k=0,\dots, K$.

If the eigenbasis is computed using, for example, homogeneous Neumann
boundary conditions, then the constant function $\phi_0(x) \equiv 1/\sqrt{|\Omega|}$ is the eigenfunction corresponding to the lowest eigenvalue $\lambda_0 = 0$ of the Laplacian, this architecture reproduces the averaging neural operator with the particular choice $T_{\ell,0} := I_{\dc\times \dc}$, and $T_{\ell,k} :=0$, for $k\ge 1$. Thus, we conclude from the universality of averaging neural operators, a result which places the
Laplace neural operator \cite{laplaceno} on a firm theoretical basis:
\begin{corollary}[Laplace Neural Operator]
\label{cor:5}
The Laplace neural operator architecture is universal in the settings of Theorem \ref{thm:universal10} and Theorem \ref{thm:universal20}.
\end{corollary}

\subsection{DeepONet}
In the course of proving universality for the averaging neural operator, outlined in Section \ref{ssec:sketch}, we prove several approximation results which demonstrate connections with other architectures, such as the Deep Operator Network (DeepONet) architecture \cite{lu_learning_2021}. We recall that a DeepONet is based on two ingredients which we now detail. Firstly, a set of linear functionals $\ell = (\ell_1,\dots, \ell_m): \cX(\Omega;\R^k) \to \R^{m \times k}$ have to be specified; often, these functionals are defined by simple point evaluation, e.g. $\ell_j(u) = u(x_j)$, $j=1,\dots, m$. Secondly, two neural networks $\beta: \R^{m \times k} \to \R^p$ (the branch net) and $\tau: \Omega \to \R^{p \times k'}$ (the trunk net) are introduced, giving rise to an operator of the form
\begin{align}
\label{eq:don}
\Psi: u \mapsto \sum_{k=1}^p \beta_k(\ell(u)) \tau_k.
\end{align}
In practice, the networks $\beta$, $\tau$ are trained from given data consisting of input- and output-pairs $(u_1,\Psi^\dagger(u_1)), \dots, (u_N, \Psi^\dagger(u_N))$, where $\Psi^\dagger: \cX \to \cY$ is the underlying (truth) operator. Given a compact set $\mfK\subset \cX$, the following approximation error is of interest:
\begin{align}
\label{eq:don-err}
\sup_{u\in \mfK}
\left\Vert
\Psi^\dagger(u) - \sum_{k=1}^p \beta_k(\ell(u)) \tau_k 
\right\Vert_{\cY}.
\end{align}
Comparing \eqref{eq:don-err} with \eqref{eq:psisimple0}, we observe that the composition $\beta_k \circ \ell: \cX \to \R$ can be thought of as providing an approximation to the nonlinear functional $\alpha_k$, whereas the trunk net $\tau_k$ approximates $\phi_k$. 

On unstructured domains, the branch net $\beta: \R^{m \times k} \to \R^p$ is often chosen as a fully connected neural network. The analysis in the present work suggests an alternative choice where $\beta\circ \ell$ is replaced by a functional $\tilde\alpha: \cX \to \R^p$, defined by
\begin{align}
\label{eq:don-alpha}
\tilde\alpha(u) := \sigma\left(
\fint_{\Omega} R(u(x),x) \, dx
\right),
\end{align}
where $R: \R^k \times \Omega \to \R^p$ is an ordinary neural network. In practice, the integral average can be replaced by a sum over a finite number of sensor points $x_1,\dots, x_m$, resulting in a mapping of the form
\[
\tilde\alpha(u) 
= 
\sigma\left(
 \textstyle{\sum_{j=1}^m R(u(x_j),x_j) }
\right).
\]
A variant of DeepONet is then obtained by defining $\tilde{\Psi}: \cX \to \cY$ by $\tilde\Psi(u) := \sum_{k=1}^p \tilde\alpha_k(u) \tau_k$. Our analysis implies that this operator learning architecture $\tilde\Psi$ is universal in the settings of Theorem \ref{thm:universal10} and \ref{thm:universal20}.

\subsection{NOMAD}
Going beyond the previous subsection we see that our analysis is also closely linked to a recent extension of DeepONets, called ``Nonlinear Manifold Decoder'' (NOMAD) \cite{seidman2022nomad}. In this approach, the functionals $\ell_1,\dots,\ell_m: \cX \to \R^{m\times k}$ and the branch net $\beta: \R^{m\times k} \to \R^p$ are retained, but the linear expansion in \eqref{eq:don} is replaced by a nonlinear neural network $Q: \R^p \times \Omega \to \R^{k'}$. The resulting NOMAD operator $\Psi: \cX \to \cY$ is defined by
\begin{align}
\label{eq:nomad}
\Psi(u)(x) := Q(\beta(\ell(u)), x).
\end{align}
Replacing the composition $\beta\circ \ell$ by the alternative functional $\tilde\alpha$ defined in \eqref{eq:don-alpha}, the NOMAD approach results in an operator $\tilde{\Psi}: \cX \to \cY$ of the form 
\[
\tilde{\Psi}(u)(x) := Q(\tilde\alpha(u), x),
\]
thus recovering the specific construction of the ANO in \eqref{eq:avg-no}.

\section{Numerical Experiments}
\label{sec:N}

\begin{figure}
    \centering
    \begin{subfigure}{0.45\textwidth}
        \includegraphics[width=\textwidth]{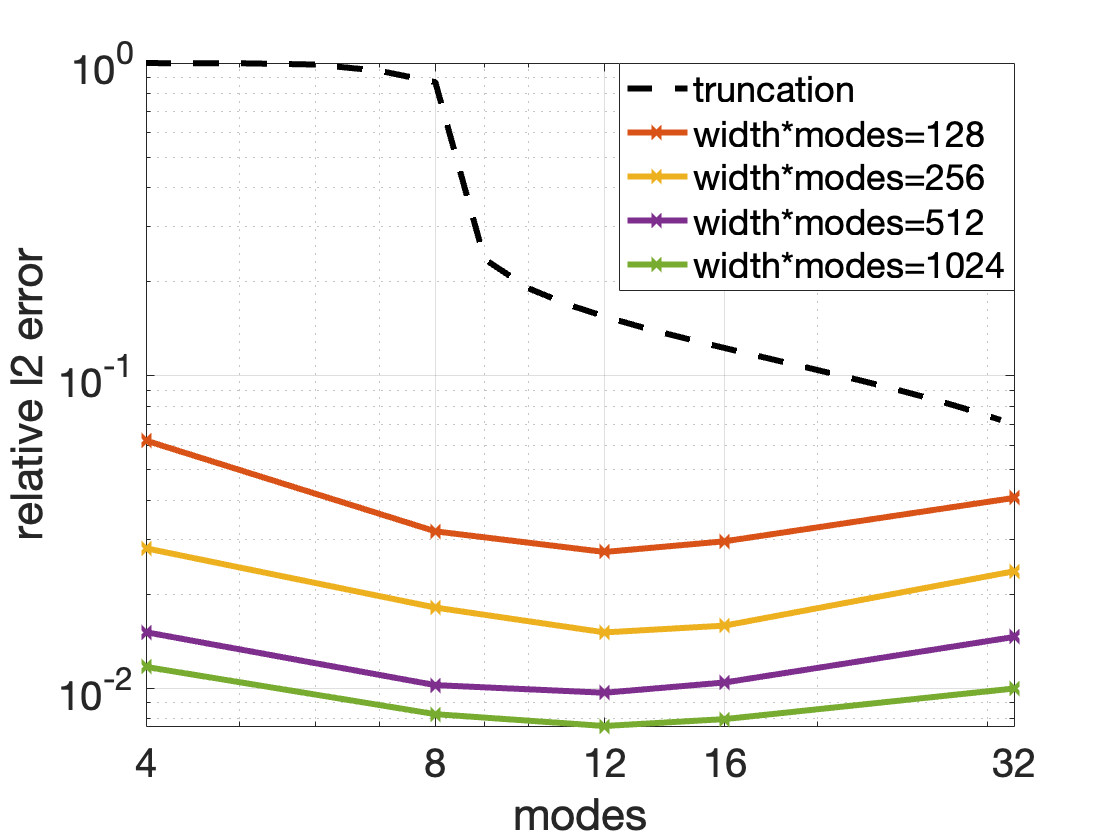}
        \caption{Helmholtz Equation}
        \label{fig:helmholtz}
    \end{subfigure}
    \begin{subfigure}{0.45\textwidth}
        \includegraphics[width=\textwidth]{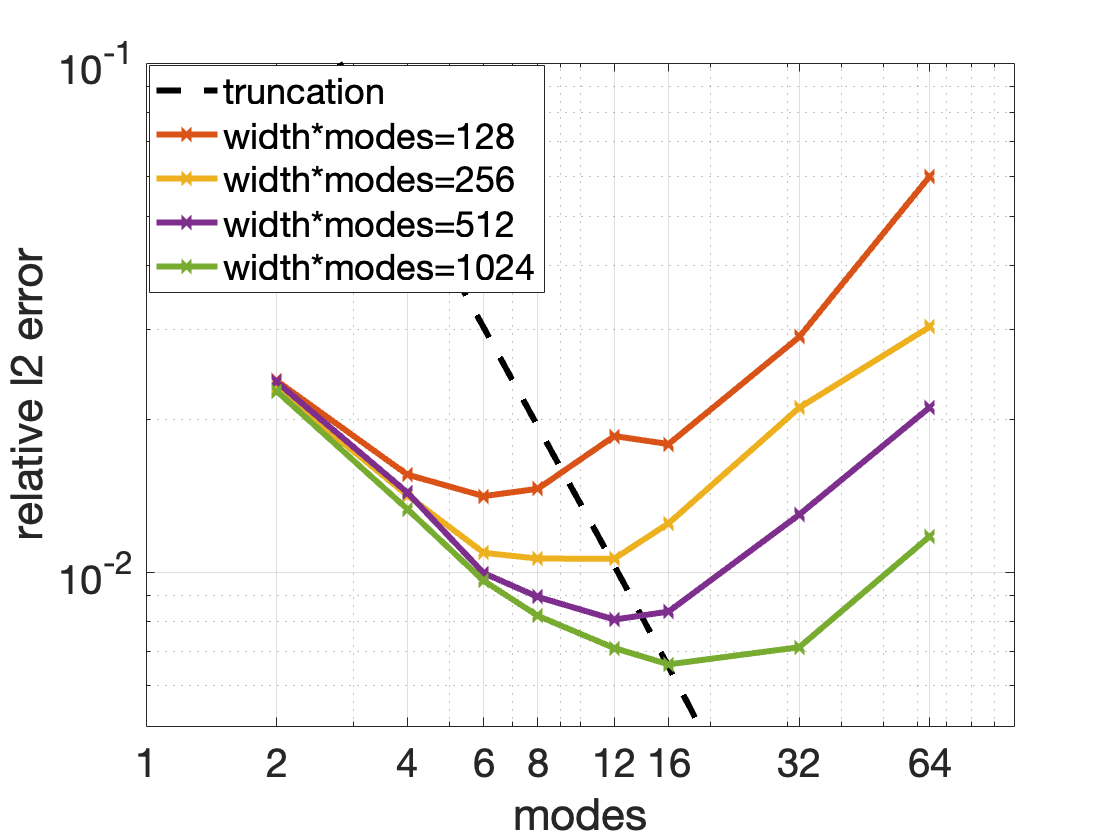}
        \caption{Darcy Equation (discontinuous)}
        \label{fig:darcy}
    \end{subfigure}
    \begin{subfigure}{0.45\textwidth}
        \includegraphics[width=\textwidth]{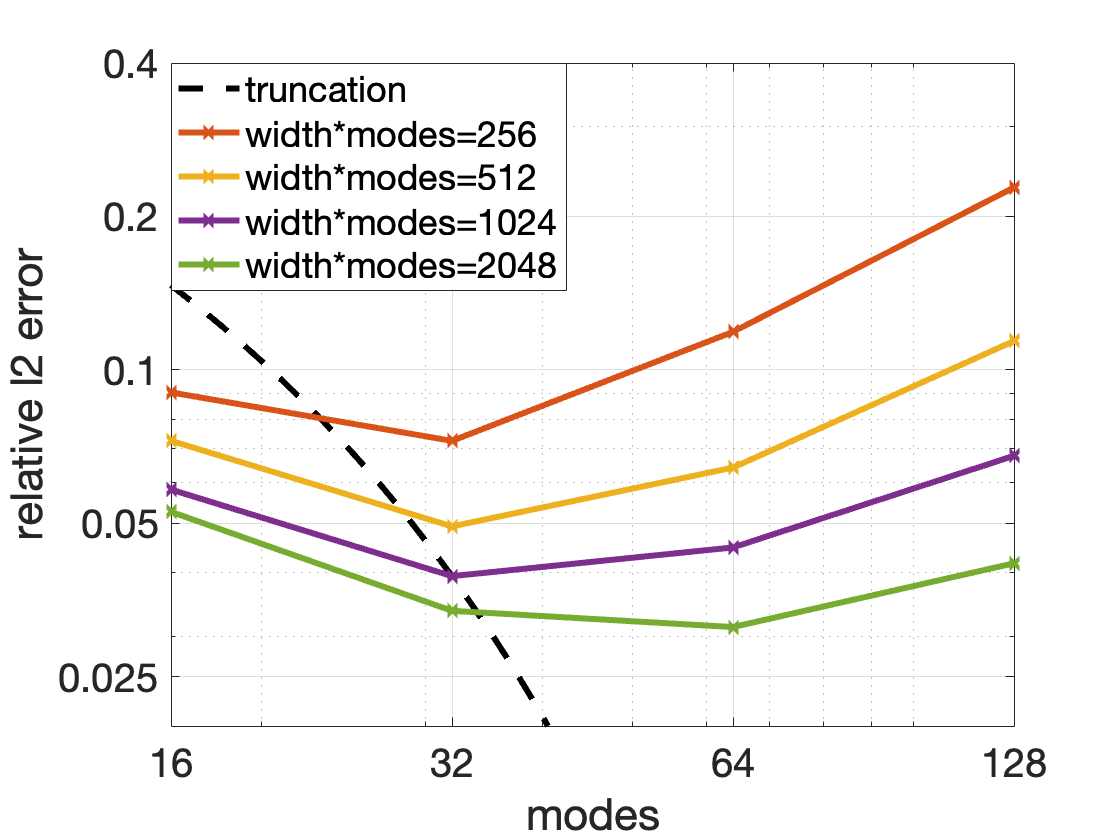}
        \caption{Kolmogorov Flow}
        \label{fig:kf}
    \end{subfigure}
        \begin{subfigure}{0.45\textwidth}
        \includegraphics[width=\textwidth]{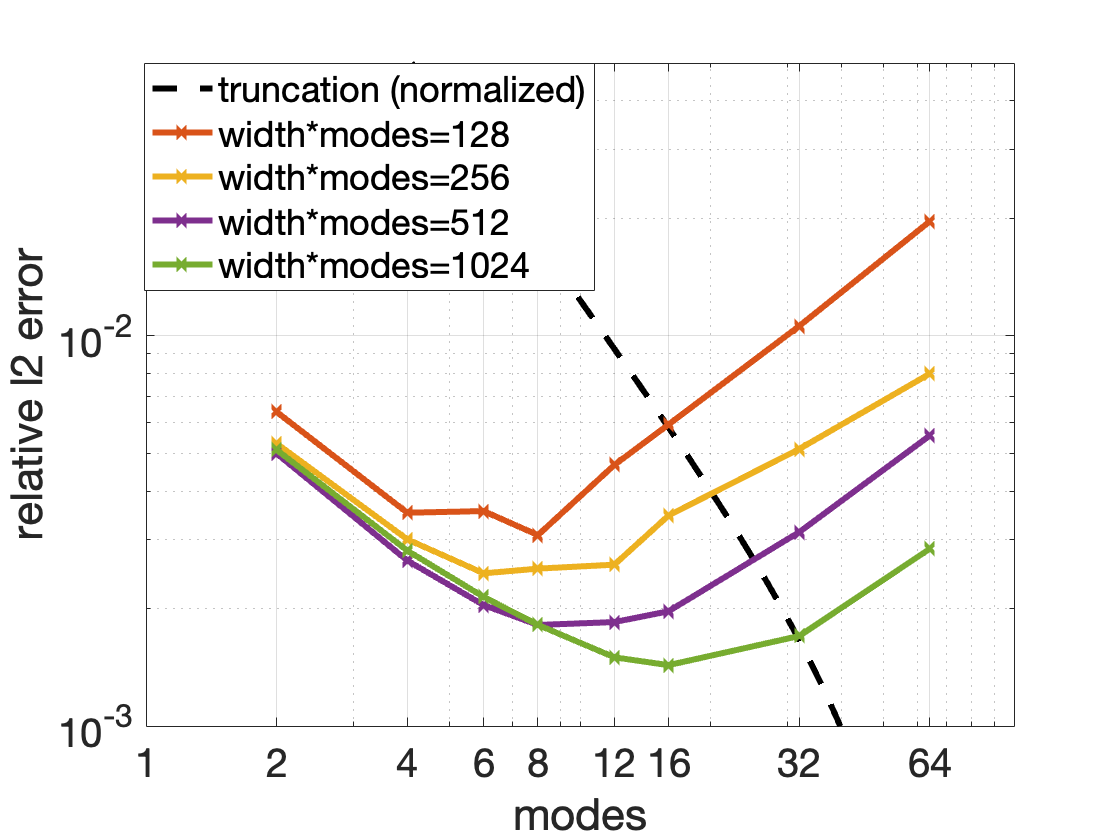}
        \caption{Darcy Equation (smooth)}
        \label{fig:darcy_lognormal}
    \end{subfigure}
    \caption{FNO model with different combinations of channel dimension (width) $d_c$ and Fourier modes $K$. In each curve, the models share roughly the same amount of model parameters
    because $d_c K$ is fixed at $C$, with $C$ changing between each curve. As $C$ increases,
    the overall error drops. The curves at any given fixed $C$ exhibit a ``U''-shape, where the valley determines the optimal choice of modes $K$. \ams{The black dotted line is the error
    given by simply truncating the truth at the given number $K$ of Fourier modes (``Fourier truncation''). Two things are notable: \zongyi{(i) for Helmholtz the optimal number of modes is fixed as computational budget increases, whilst for  the two Darcy problems and Kolmogorov flow  it grows;} (ii) in all examples the trained FNO is able to considerably improve on the Fourier truncation, when smaller numbers of Fourier modes are used; this suggests the
    nonlinear approximation theoretic mechanisms at play and resulting from the architecture.}}
    \label{fig:exp}
\end{figure}

There are two types of numerical experiments that are suggested by the analysis
in this paper: (i) to deploy the NNO for problems in
non-periodic geometries, and in particular to explore choices of problem-adapted functions $\phi_{\ell,m}, \psi_{\ell,m}$ in \eqref{eq:gexp}, aiming to develop intuition as to when 
it is an effective
operator approximator; (ii) to determine the optimal distribution of parameters to obtain a given approximation error at least cost. We defer experiments of type (i) for future work and concentrate here on experiments of type (ii).

\ams{To carry out experiments of type (ii) we work in the setting outlined at the start of subsection \ref{ssec:NNO}. The channel width is $d_c$ and we employ $L=4$ layers. The domain $\Omega=[0,1]^d$ and in all examples considered $d=2.$ The layer takes the specific form
\begin{align*}
(\cL_{\ell} v)(x) := 
\sigma\left(
W_\ell v(x) + b_\ell + \sum_{k \in \mathbb{K}} \langle T_{\ell,m} v, \exp\bigl(2\pi i \langle k,x\rangle\bigr) \rangle_{L^2(\Omega;\mathbb{C}^{d_c})} \exp\bigl(2\pi i \langle k,x\rangle\bigr)(x)
\right),
\end{align*}
where $\mathbb{K}$ is a finite square lattice set of cardinality $K^2$, centred at the origin.
Thus $|k|_{\infty} \le K.$
In our experiments we fix a measure of the total number of parameters: $d_c \times K.$
We then identify the optimal choice of number of basis functions $K$, in terms of the achieved test accuracy,
for fixed total number of parameters.}
We conduct experiments on three sets of partial differential equations to study the
balance between number of channels and number of Fourier modes: the Helmholtz Equation for wave-propagation in
a heterogeneous medium, the Darcy Equation for porous medium flow, for both smooth and piecewise constant coefficients, and the 
forced two-dimensional Navier-Stokes equation, more specifically a Kolmogorov Flow. 

For all three test cases, we use a 4-layer FNO model with the GeLU activation function, as in \cite{li_fourier_2021}, and  a cosine annealing optimizer \cite{loshchilov2016sgdr}. In the Helmholtz and Darcy equations, we normalize the input and output functions by subtracting the pointwise mean and dividing by
the pointwise standard deviation, with mean and pointwise variance computed from the training dataset; however
in the Kolmogorov Flow, we do not use normalization.
The results show that, in the range of parameter scenarios we consider, the Helmholtz Equation and Kolmogorov Flow are optimized by fixing the number of basis functions; in contrast the Darcy problem is optimized by increasing the total number of basis functions, given a fixed budget of total number of parameters. These results resonate with the
different proofs of universal approximation for the FNO, the original one using
an increasing number of Fourier modes to achieve smaller error \cite{kovachki_universal_2021} 
and the proof in this paper using only a single Fourier
mode found via integration against the constant function; more generally the
results suggest the need for deeper analysis of the deployment of parameters in
neural operators to obtain the optimal error/cost tradeoff.

\subsection{Helmholtz Equation}
We consider the Helmholtz Equation as defined in \cite{de2022cost}.
\sam{
Specifically, we consider the following PDE,
\begin{align*}
\left( - \Delta - \frac{\omega^2}{c^2} \right) u &= 0, \quad \text{in } \Omega,
\end{align*}
with frequency $\omega >0$ and variable coefficient field $c=c(x)$, on a square domain $\Omega = [0,1]^2$. We impose Neumann boundary conditions,
\[
\frac{\partial u}{\partial n} = 0, \text{ on }\partial \Omega_1, \partial \Omega_2, \partial \Omega_4, \quad
\frac{\partial u}{\partial n} = u_N, \text{ on } \partial \Omega_3.
\]
 Here $\partial \Omega_1$, $\partial \Omega_2$, $\partial \Omega_3$ and $\partial \Omega_4$ denote the lower, right, upper and left edges of the square. We fix $u_N(x) = 1_{[0.35\le x\le 0.65]}$.
 } 
The frequency $\omega$ is fixed at $10^3$. 
The target operator maps the wave speed field to the disturbance field $\Phi^{\dagger}:c \mapsto u$. The probability measure on the inputs is a pointwise nonlinear
transformation of a Gaussian random field: $c(x) = 20 + {\rm tanh}(\tilde{c}(x)), \tilde{c} \sim \mathcal{N}(0, 
(-\Delta + 9)^{-2})$ where the Laplace operator is equipped with homogeneous Neumann boundary conditions. We use $10^4$ instances of data pairs to train the FNO model, each with resolution $10^2\times 10^2$.
\zongyi{We use Adam optimizer and cosine annealing learning rate scheduler, with initial learning rate $1 \times 10^{-3}$ and weight decay $1 \times 10^{-4}$.}

\subsection{Darcy Equation}
The Darcy equation is widely studied in the operator learning literature and
we use the set-up employed in \cite{bhattacharya_model_2021, anandkumar_neural_2020}. \ams{We consider the following PDE,
\[
-\nabla \cdot (a \nabla u) = 1,
\]
posed in $\Omega=[0,1]^2$ and subject to homogeneous Dirichlet boundary conditions.
}
The target operator maps the coefficient field (permeability) to the solution field
(pressure) $\Phi^{\dagger}:a \mapsto u$.  We use $10^3$ data pairs to train the FNO model, each with resolution $141\times 141$. \zongyi{We use Adam optimizer and cosine annealing learning rate scheduler, with initial learning rate $1 \times 10^{-3}$ and weight decay $1 \times 10^{-4}$.}

We study two cases of the Darcy equation, differing in our choice of random input fields: 
piecewise-constant coefficients and log-normal coefficients.
For the piecewise-constant case, the coefficient functions $a$ are samples from Gaussian random field $\tilde{a} \sim \mathcal{N}(0,(-\Delta + 9)^{-2})$ again with homogeneous Neumann
boundary conditions. We apply a pointwise nonlinearity which assigns value $4$ or $12$
depending on the sign of the input: $a(x) = 4$ if $\tilde{a}(x) \leq 0$, $a(x) = 12$ if $\tilde{a}(x) > 0$. For the log-normal case, the coefficient function $a$ is sampled from a lognormal distribution, which is equivalent to the exponential of 
the same Gaussian random field:  $a(x) = \exp(\tilde{a}(x))$. 

\subsection{Kolmogorov Flow}
Our final example is  the Kolmogorov Flow with set-up as defined in \cite{li2021markov}. 
For this problem, we solve the periodic Navier-Stokes equation on a
torus \footnote{\ams{A straightforward rescaling of the domain used to define the neural operator above, from $[0,1]^d$ to $[0,2\pi]^2$, is employed.}}$\T^2 \simeq [0,2\pi]^2$
\[
\partial_t u + u \cdot \nabla u + \nabla p = \frac{1}{Re} \Delta u + \sin(ny) \hat{x},
\quad \nabla \cdot u = 0,
\]
with Reynolds number $Re = 5 \times 10^2$; this Reynolds number
leads to a challenging example in which more high-frequency structures are present in comparison to the previous cases considered in the literature. \ams{We define the vorticity
$\omega = \mathrm{curl} \, u$ and seek to learn the map from the initial vorticity field to the vorticity field one time unit later:  $\Phi^{\dagger}: \omega(t) \mapsto \omega(t+h)$. \zongyi{We choose the time unit $h=0.1$, noting that learning the solution operator on long time-horizons can be difficult
because of the chaotic behavior present at this Reynolds number. \footnote{\zongyi{Short-time solution operators can be composed using the semigroup property to make longer-term statistical predictions \cite{li2021learning}.}}}
The equations may be formulated using the vorticity-stream function formulation and the numerical
method we employ to generate the data is a pseudo-spectral method applied to this formulation.
The dataset of input-output pairs is computed from $100$ trajectories, with $80$ used for training and $20$ 
for testing. Each trajectory consists of $500$ time units;  the 
burn-in phase, during which the dynamics evolves towards to attractor, is removed from the trajectory.
We sample the initial conditions from a Gaussian random field of similar form to the two preceding examples,
but with $-\Delta$ subjected to periodic boundary conditions.
In total, we use $4 \times 10^4$ data pairs to train the FNO model, each with resolution $128\times 128$.
Note that the data is not i.i.d.}
\zongyi{We use the Adam optimizer and cosine annealing learning rate scheduler, with initial learning rate $5 \times 10^{-4}$ and weight decay $1 \times 10^{-4}$.}

\subsection{Results}

We study the accuracy of the trained FNO for different combinations of channel dimension (or width) $d_c$ and Fourier modes $K$, with a given budget of the total number of parameters.
Since the number of parameters scales as $O(d_c^2 K^2)$ for these 2-dimensional problems, we fix $d_c K = C$ as a convenient proxy to control the sizes of our models. The main theoretical result in this paper demonstrates that, in principle, universality can be achieved by simply retaining a single Fourier mode. In practice, it may be beneficial to include a larger number of modes, and we are interested in the \emph{optimal} choice of $K$ (and hence $d_c$), given a fixed model size $C$. To study this refined question empirically, we consider different model sizes, with $C=128, 256, 512, ...$, and scan over increments of the Fourier mode cut-off $K = 2, 4, 8, ...$.
The results of this empirical study are collected in Figure \ref{fig:exp}, for the four settings summarized above; in each case, the relative $L^2$-errors are plotted as a function of the number of the Fourier truncation parameter $K$, and each solid curve represents the result of
computations with a fixed number $C$ of model parameters. As a baseline, the dashed black line shows the average relative $L^2$-error that is achieved by a direct Fourier truncation of the reference output functions, as a function of $K$, \sam{i.e. the average error over the test distribution that is obtained by projecting the reference solution of the PDE to an expansion of its Fourier modes with wavenumbers $|k|_\infty \le K$.}

Overall, we observe that models with more parameters perform better. Moreover, for a fixed model size $C$, the solid curves describe a ``U'' shape as a function of $K$, implying that an optimal choice of the Fourier cut-off $K$ and channel dimension $d_c$ exists, at each given model size $C$. 
For the Helmholtz equation the optimal choice of Fourier cut-off $K$ is fixed across different model sizes, within the resolution of our experiments. The trained FNO for the Helmholtz Equation performs best with $K=12$. For the Kolmogorov Flow, the optimal choice is \zongyi{$K=32$ for $C=256,512,1024$ except for $C=2048$ where the optimal mode is $K=64$. The optimal model is the smallest $K$ larger than the corresponding Fourier truncation mode.}
A priori, our theoretical results are not informative about the \emph{optimal} distribution of parameters. However, they show that even a fixed choice of Fourier modes $K$, independent of the model size, is not in contradiction with the observed convergence to the underlying operator. This empirical observation would have been difficult to reconcile with previous universal approximation results, which require $K$ to grow without bound. 

For the Darcy Equation, the optimal choice of modes clearly shifts to the right with increasing model size. When $C=d_c K = 128$, the optimal choice of modes is $K=6$. When $d_c K = 1024$, the optimal choice of modes is $K=16$. Thus, in this case, it appears to be optimal to invest additional degrees of freedom in a combination of increasing both $K$ and $d_c$. 
An interesting empirical observation in Figure \ref{fig:exp} is that the trained FNO often beats the corresponding Fourier truncation of the true solution with the same number of modes. This is particularly apparent for the Darcy Equation, where even the FNO with a cut-off of only $K=2$ is able to achieve very accurate results, going considerably beyond the accuracy of the corresponding Fourier truncation baseline. This strongly indicates that  the trained FNO discovers a non-linear mechanism to generate a well-calibrated set of higher-order Fourier modes, beyond $K=2$. Thus, non-linear composition combined with non-locality of the lowest Fourier modes enables accurate results even with very small $K$; this is the basic mechanism which underlies our main universality result. The empirical findings thus indicate the practical relevance of the basic mechanism identified in the present work, with implications beyond universal approximation.

\begin{figure}[h]
    \centering
    \begin{subfigure}{0.45\textwidth}
        \includegraphics[width=\textwidth]{figures/helm.png}
        \caption{Helmholtz Equation}
        \label{fig:helm}
    \end{subfigure}
    \begin{subfigure}{0.45\textwidth}
        \includegraphics[width=\textwidth]{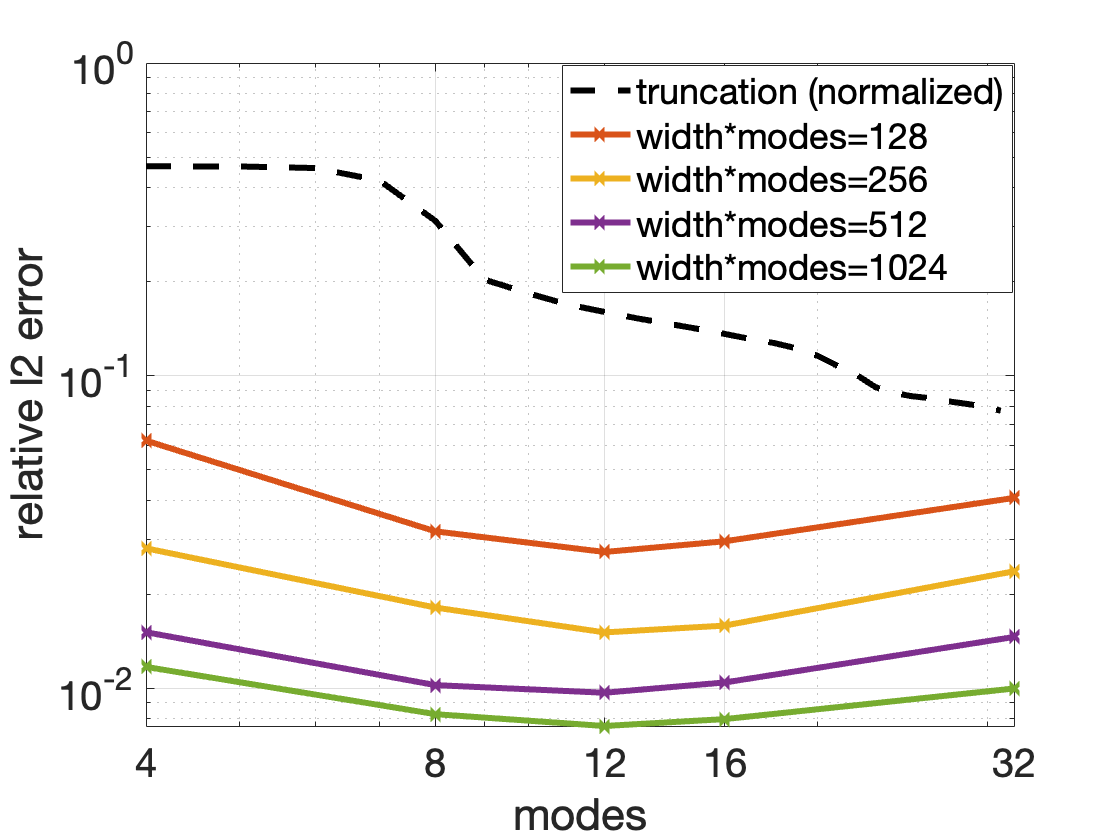}
        \caption{Helmholtz Equation (normalized)}
        \label{fig:helm_normalized}
    \end{subfigure}\\
    \hspace{0.5cm}
    \begin{subfigure}{0.4\textwidth}
        \includegraphics[width=\textwidth]{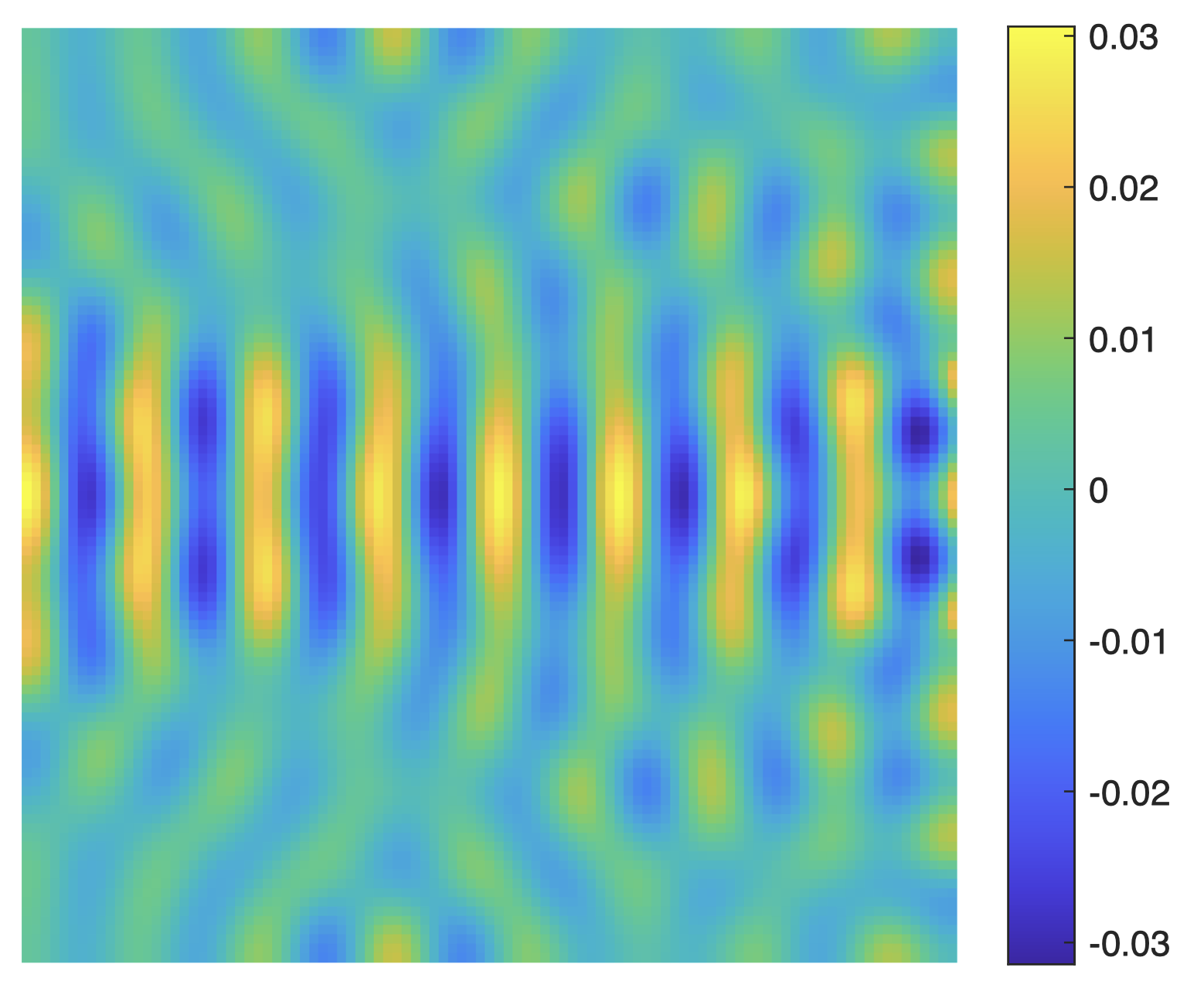}
        \caption{Mean of Helmholtz}
        \label{fig:helm_mean}
    \end{subfigure}
    \hspace{0.5cm}
    \begin{subfigure}{0.4\textwidth}
        \includegraphics[width=\textwidth]{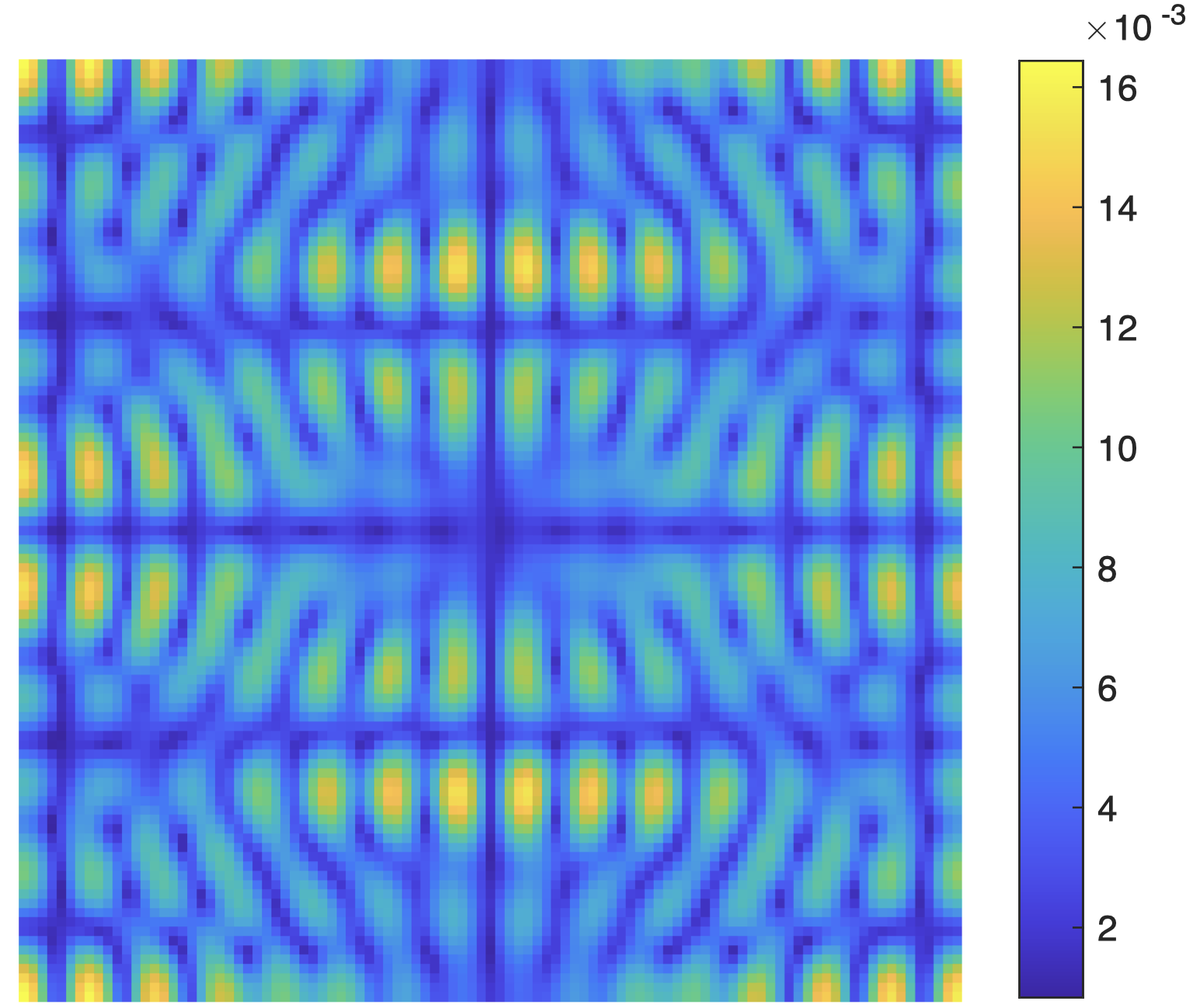}
        \caption{Standard deviation of Helmholtz }
        \label{fig:helm_std}
    \end{subfigure}
    \caption{The effect of normalization demonstrated for the Helmholtz Equation example. (a) Error rates as shown in Figure \ref{fig:exp}a). (b) Same as Figure \ref{fig:exp}a) but with error rates of the Fourier truncation including the effect of normalization. (c) and (d) show the mean and standard deviation, respectively, as
    used in the normalization process. Normalization assists in capturing lower modes but it does not affect the higher modes.}
    \label{fig:exp_normalization}
\end{figure}

We include one final set of experiments designed to check that the effect of normalization (applied on the Helmholtz and both Darcy examples) is not responsible for effects apparent in  Figure \ref{fig:exp}.
Indeed we show that the desirable approximation effects we observe there cannot simply be attributed to the pre- and post-processing of the inputs and outputs of the neural network to ensure (empirical) mean zero and pointwise standard deviation one. To verify this assertion we show, in Figure \ref{fig:exp_normalization}, the error found simply by Fourier truncation, with and without the normalization applied. 
In Figure \ref{fig:exp_normalization} (a), the truncation is done without the normalization. 
In Figure \ref{fig:exp_normalization} (b), the truncation is done with the normalization: the solution function is first normalized by the pointwise mean and standard deviation, then truncated to the requisite number of Fourier modes, and finally denormalized, mimicking the procedure used in the model training.
The truncation with normalization has a smaller error rate on the lower frequency modes, but the error at higher wavenumbers is similar to that arising without the truncation. The pointwise mean and standard deviation are shown in \ref{fig:exp_normalization} (c, d), which demonstrates that the normalization captures some lower frequency structures.
This experiment shows that normalization help to express the lower frequency components, but it does  
not account for the desirable approximation properties observed for the neural operator in  Figure \ref{fig:exp}.
The Figure \ref{fig:exp_normalization} is for the Helmholtz equation, but similar effects are observed
for the Darcy flow.

\section{Discussion}
\label{sec:D}

\sam{
This paper considers a general framework of nonlocal neural operators (NNO), which includes known frameworks such as the Fourier neural operator (FNO). In a nutshell, we derive sufficient (minimal) conditions under which specific instances of this general framework are guaranteed to possess a universal approximation property. To answer this question, the generic NNO architecture is reduced to its simplest possible form, the averaging neural operator (ANO). This ANO is based on only two minimal ingredients, both of which are indispensable for operator learning: \emph{nonlinearity}, by composition with ordinary neural networks, and \emph{nonlocality}, by averaging. Despite the simplicity of the resulting architecture, we show that the averaging neural operator (ANO) possesses a universal approximation property for a large class of operators, cf. Theorems \ref{thm:universal10} and \ref{thm:universal20}.
}

\ams{The methodology provides a unifying perspective on universal approximation, across a wide-range of neural network architectures, and also suggests
new architectures, to which universal approximation will also apply.
It allows us to unify much of the theoretical analyses of, for example, the general NO, FNO, low-rank NO, wavelet NO and Laplace-NO; the ANO is employed as an analysis tool to derive new universality results for all of these architectures, leading to Corollaries \ref{cor:1}--\ref{cor:5}. These corollaries improve on known universal approximation results for general NOs and FNOs, and provide a first rigorous theoretical underpinning for the methodology in the case of low-rank, wavelet and Laplace NOs. Furthermore we
believe that our general universality result for the ANO will serve as a useful reference to prove universality for many other emerging neural operator architectures.} In this context, we also point out close links of the ANO introduced in the present work with other recent proposed neural operator architectures, including the DeepONet and NOMAD architectures. \sam{In fact, our proof of universality of the ANO is based on identifying an inherent encoder-decoder structure in this architecture, and proving universality of such encoder-decoders.}

At a more fundamental level, the present work aims to shed new light on the role of nonlocality in operator learning for a general class of neural operators. It is shown that even very simple nonlocality in the form of an averaging operation is sufficient for the universality of neural operator architectures. As a consequence, we show that neural operator architectures possess a universal approximation property even if the nonlocality in the hidden layers is not capable of exploring the space of all integral operators, thus providing considerable scope for the introduction of nonlocality in new ways which may not rely on expansions in complete (e.g. Fourier or wavelet) bases. As a byproduct, we also further deepen our understanding of the Fourier neural operator; in particular, the present work strongly indicates that the importance of the Fourier transform for FNOs may be mainly in providing nonlocality to the architecture, with the particular choice of a Fourier basis playing only a subordinate role.

The present work provides substantial motivation to further study the optimal combination of nonlinearity and nonlocality in operator learning. First empirical results in this direction are presented in Section \ref{sec:N}, focusing on the Fourier neural operator. Here, the relevant parameters are the channel dimension and the number of Fourier modes, corresponding to different distributions of nonlinearity and nonlocality in the architecture. Our empirical results show that the optimal distribution of parameters is likely problem dependent, but strongly indicate that the basic mechanism underlying our proof of universality is also relevant in practice; further numerical results, in tandem with extended analysis, will be necessary to gain deeper insight into the optimal trade-off, in terms of cost versus accuracy, between non-linearity and non-locality. 

\sam{
It is also interesting to compare and contrast the ANO architecture with convolutional neural networks (CNN) and convolutional neural operators (CNO). The ANO employs a global average, applied to a non-linear transformation of the input function, to encode relevant features; in contrast, convolutional architectures like CNN and CNO are based on a composition of non-linear transformations, convolutions with localized kernels, and potentially interspersed with max-pooling (or average pooling) operations. From the point of view of universality, we note that the ANO encoder defined via a global average in the present work, could equally well have been replaced by local averaging over a set of patches covering the domain, providing a link to average pooling in CNNs. One major difference between ANO and CNN/CNO is that the layerwise nonlocality is not global for the latter, thus requiring a minimal depth to communicate nonlocal information across the entire domain. We believe it would be interesting to extend the analysis of the present work to CNOs, in the future.}

The fact that universality can be obtained even with a simple global average as the only nonlocal ingredient in the hidden layers opens the way for the exploration of alternative neural operator architectures with \emph{guaranteed universality}, relying on different choices of how nonlocality is introduced. The results of this work are also particularly relevant for future extensions of neural operator architectures to problems with non-periodic geometries and mappings between functions on different domains. We plan to expand on this topic, and the potential benefits of specific choices, in the future. 

\sam{In another direction, it would also be interesting to derive quantitative error estimates for ANO, to complement the qualitative analysis of the present work. Our initial numerical experiments clearly indicate practical benefits, in terms of the accuracy of the trained models, of introducing nonlocality beyond simple averaging; it would be interesting to see if this is also born out by a quantitative analysis. However, in view of the great (theoretical) approximation power of neural networks as function approximators, here used to define the encoder and decoder in the ANO, it is possible that an explanation of the observed empirical differences would require going beyond a purely approximation theoretic viewpoint, and taking into account their training by gradient-based optimization on finite data.}

\vspace{0.1in}
\noindent{\bf{Acknowledgments}}
The work of SL is supported by Postdoc.Mobility grant P500PT-206737 from the Swiss National Science Foundation.
ZL is supported in part by the PIMCO Fellowship and Amazon AI4Science Fellowship. 
AMS is grateful for support through a Department of Defense Vannevar Bush Faculty Fellowship 
and from the Air Force Office of Scientific Research under MURI award number FA9550-20-1-0358 (Machine Learning and Physics-Based Modeling and Simulation).

\vspace{0.1in}

\bibliographystyle{abbrv}
\bibliography{bibliography}

\raggedbottom
\newpage

\appendix

\section{Detailed Proof Of Universal Approximation For ANO}
\label{app:pf-univ}

\subsection{Analysis Ingredients: Extension And Mollification}
\label{app:moll}

\paragraph{Lipschitz Domains}

We recall that a subset $\Omega \subset \R^d$ is a \define{bounded Lipschitz domain}, if $\Omega$ is open, the closure $\overline{\Omega}$ is compact, and the boundary $\partial \Omega = \overline{\Omega} \setminus \Omega$ is at least ``Lipschitz regular'', in the sense that it can locally be thought of as the graph of a Lipschitz map. We refer to \cite[Section 2.1]{ErnGuermond2016} for a rigorous mathematical definition. Lipschitz domains cover most domains encountered in physical applications. In particular, any bounded domain with piecewise smooth boundary is a Lipschitz domain.

\paragraph{Function Extension for Lipschitz Domains}

In our proofs, it will sometimes be convenient to extend a given function $u: \Omega \to \R$, defined on a domain $\Omega \subset \R^d$, to a function $u: \R^d \to \R$, defined on all of $\R^d$. The following lemma shows that this is possible, while preserving smoothness of $u$, and relies on a classical result of Stein \cite[Chapter 6, Theorem 5]{stein1970singular}:

\begin{lemma}[{Periodic extension operator, see e.g. \cite[Lemma 41]{kovachki_universal_2021}}]
\label{lem:periodic_extension}
Let \(\Omega \subset \R^d\) be a bounded Lipschitz domain. There exists a continuous, linear operator $\cE : W^{s,p}(\Omega) \to W^{s,p}_\per(B)$ for any $s \geq 0$ and $p\in [1,\infty]$, where $B \subset \R^d$
is a bounded hypercube containing $\Omega \subset B$,
such that for any $u \in W^{s,p}(\Omega)$:
\begin{enumerate}
	\item $\mathcal{E}(u)|_\Omega = u$;
	\item $\mathcal{E}(u) \in W^{s,p}_\per(B)$ is periodic on \(B\) (including its derivatives). 
\end{enumerate}
Furthermore, $\cE$ maps continuously differentiable functions to continuously differentiable functions, i.e. $\cE(C^s(\bar{\Omega})) \subset C^s_\per(B)$ and hence defines a continuous mapping $\cE: C^s(\bar{\Omega}) \to C^{s}_\per(B)$.
\end{lemma}

\paragraph{Mollification (Smoothing) of Functions on Lipschitz Domains}
Recall that there exists a smooth mapping (a mollifier) $\rho: \R^d \to \R$, 
with properties $\rho(x)\in [0,1]$ for all $x \in \R^d$, $\rho(0) = 1$, and $\rho(x) = 0$ for $|x|\ge 1$. Furthermore we can normalize $\rho$ to enforce $\int_{\R^d} \rho(y) \, dy = 1$. Any such $\rho$ defines a family of functions $\rho_\delta(x) := \delta^{-d} \rho(x/\delta)$, supported in a $\delta$-ball around the origin. Fixing such a family, we recall that the $\epsilon$-mollification of a function $u: \R^d\to \R$, $u\in L^1(\R^d)$ is defined by a convolution $u_\delta(x) := (u\ast \rho_\delta)(x)$, i.e.
\[
u_\delta(x) = \int_{\R^d} u(x-y) \rho_\delta(y) \, dy.
\]
It is well-known that $u_\delta$ is a smooth function for $\delta > 0$, and that $u_\delta \to u$ in spaces such as arbitrary Sobolev spaces $W^{s,p}(\R^d)$, or spaces of continuously differentiable functions $C^s(\R^d)$, with respect to the respective norms. Mollification as defined above is a useful tool in analysis, but is not well-adapted to bounded domains. The papers \cite{blouza2001up,ErnGuermond2016} develop a useful variant of mollification for bounded Lipschitz domains $\Omega \subset \R^d$, where additional care is needed to deal with the behavior near the boundary.

The following results follow from the proof of \cite[Theorem 2.4]{blouza2001up}, or can alternatively be obtained from the construction and arguments in \cite{ErnGuermond2016}:
\begin{lemma}[Adapted mollification in a Lipschitz domain]
\label{lem:moll}
Let $\Omega \subset \R^d$ be a bounded Lipschitz domain. There exists a one-parameter family $\cM_\delta$ of ``mollification'' operators, indexed by $\delta \ge 0$ and defining a linear mapping $\cM_\delta: L^1(\Omega; \R^k) \to L^1(\Omega;\R^k)$ for $\delta > 0$, such that:
\begin{enumerate}
\item \sam{For $\delta = 0$, we have $\cM_{0}u = u$ for all $u\in L^1(\Omega;\R^k)$; for $\delta > 0$, $\cM_{\delta}$ is smoothing, in the sense that it defines a mapping $\cM_\delta: L^1(\Omega;\R^k) \to C^\infty(\bar{\Omega};\R^k)$.}
\item For any fixed $\delta > 0$, integer $s,r \ge 0$, the mapping,
\[
\cM_\delta: C^s(\bar{\Omega};\R^k) \to C^r(\bar{\Omega};\R^k),
\]
is continuous. Furthermore, if $r\le s$, then the operator norm is uniformly bounded,
that is
\[
\sup_{\delta > 0} \Vert \cM_\delta \Vert_{C^s \to C^r} < \infty.
\]
\item If $s\ge 0$ and $\mfK \subset C^s(\bar{\Omega}; \R^k)$ is a compact subset, then \sam{for any $\delta_0 \ge 0$,}
\[
\lim_{\delta \to \delta_0} \sup_{u\in \mfK} \Vert \cM_{\delta_0} u - \cM_\delta u \Vert_{C^s} = 0.
\]
\item For fixed $\delta > 0$, integer $s,r \ge 0$, and $p\in [1,\infty)$, the mapping
\[
\cM_\delta: W^{s,p}(\Omega;\R^k) \to W^{r,p}(\Omega;\R^k)
\]
is continuous.
\item If $s \ge 0$, $p \in [1,\infty)$ and $\mfK \subset W^{s,p}({\Omega}; \R^k)$ is a compact subset, \sam{and $\delta_0 \ge 0$,} then 
\[
\lim_{\delta \to \delta_0} \sup_{u\in \mfK} \Vert \cM_{\delta_0} u - \cM_\delta u \Vert_{W^{s,p}} = 0.
\]
\end{enumerate}
\end{lemma}

\begin{remark}
At first sight, one might think that the result of Lemma \ref{lem:moll} could easily be obtained by extending the input function $u: \Omega \to \R^k$, to a function $u: \R^d \to \R^k$ by setting $u(x) = 0$ for $\R^d\setminus \Omega$, and then use standard mollification on this extension to obtain a smooth approximant $u_\delta: \Omega \to \R^k$. While this approach ensures smoothness of the mollified function $u_\delta$, it does not ensure convergence in spaces defined with respect to the supremum norm, i.e. for $u\in C^s$ we generally have $\Vert u_\delta - u\Vert_{C^s} \not\to 0$ as $\delta \to 0$. The main technical issue is that extension by zero produces a jump discontinuity at the boundary; hence special care needs to be taken near the boundary \cite{blouza2001up,ErnGuermond2016}.

The proof is based on the boundary-adapted construction of mollification in \cite[eq. (3.4a)]{ErnGuermond2016}, and only requires arguments in real analysis. 
We will not provide a detailed proof here. 
\end{remark}

In addition to Lemma \ref{lem:moll}, we will furthermore need the following technical lemma:
\begin{lemma}
\label{lem:Kdelta}
Fix $s\ge 0$. Let $\mfK\subset C^s(\bar{\Omega};\R^k)$ be compact. Then for any $\delta > 0$, the set 
\begin{align}
\label{eq:Kdelta}
\mfK_\delta :=
\bigcup_{ \sam{0\le \delta' \le \delta}} \set{\cM_{\delta'} u}{u\in \mfK},
\end{align}
is also compact in $C^s(\bar{\Omega};\R^k)$.
\end{lemma}

\begin{proof}
Let $v_1,v_2,\dots $ be an arbitrary sequence in $\mfK_\delta$. It suffices to prove that $v_j$ possesses a convergent subsequence $v_{j_\ell} \to v\in \mfK_\delta$. By definition of $\mfK_\delta$, there exists a sequence $u_1,u_2,\dots \in \mfK$, and $\delta_1,\delta_2,\dots \in (0,\delta]$, such that $v_j = \cM_{\delta_j} u_j$ for all $j\in \N$. Since $\mfK$ is compact, there exists a convergent subsequence $u_{j_\ell}\to u \in \mfK$. Furthermore, extracting another subsequence if necessary  (not reindexed), we may assume that $\delta_{j_\ell} \to \delta_\infty \in [0,\delta]$ converges to a limit. If $\delta_\infty > 0$, let $v := \cM_{\delta_\infty} u$. If $\delta_\infty = 0$, we set $v := \cM_{0}u = u$. We note that in either case, we have $v\in \mfK_\delta$. We claim that $v_{j_\ell}\to v$. To see this, note that 
\begin{align*}
\limsup_{\ell\to \infty} \Vert v_{j_\ell} - v \Vert_{C^s}
&= \limsup_{\ell\to \infty} \Vert \cM_{\delta_{j_\ell}} u_{j_\ell} - \cM_{\delta_\infty} u \Vert_{C^s}
\\
&\le \limsup_{\ell\to \infty} \Vert \cM_{\delta_{j_\ell}} u_{j_\ell} - \cM_{\delta_{j_\ell}} u \Vert_{C^s}
\\
&\qquad + \limsup_{\ell\to \infty} \Vert \cM_{\delta_{j_\ell}} u - \cM_{\delta_\infty} u \Vert_{C^s}
\\
&\le \limsup_{\ell\to \infty} \Vert \cM_{\delta_{j_\ell}} \Vert_{C^s\to C^s} \Vert u_{j_\ell} - u \Vert_{C^s}
\\
&\qquad + \limsup_{\ell\to \infty} \Vert \cM_{\delta_{j_\ell}} u - \cM_{\delta_\infty} u \Vert_{C^s}.
\end{align*}
For fixed $u\in C^s$, we have $\cM_{\delta_{j_\ell}} u \to \cM_{\delta_\infty} u$ in $C^s$, and hence the second term converges to zero. Furthermore, by Lemma \ref{lem:moll} (1), the one-parameter family $\cM_{\delta'}: C^s \to C^s$ consists of uniformly bounded operators, such that 
\[
\sup_{\ell \in \N}
\Vert \cM_{\delta_{j_\ell}} \Vert_{C^s\to C^s} 
\le 
\sup_{0< \delta'\le \delta} \Vert \cM_{\delta'} \Vert_{C^s\to C^s} < \infty.
\]
From the convergence $u_{j_\ell} \to u$ in $C^s$, it thus follows that the first term above converges to zero. Thus, $v_{j_\ell} \to v$ as claimed. This shows that $\mfK_\delta$ is sequentially compact.
\end{proof}

\subsection{Universal Approximation Of Neural Networks}
\label{app:nn-univ}

It is well-known \cite[Thm. 4.1]{pinkus1999approximation} that a neural network architecture is universal in the class of $C^{s}$-functions between Euclidean
vector spaces, provided that the activation function $\sigma$ is nonpolynomial and sufficiently smooth,  $\sigma \in C^{s}$. This in turn implies universality of neural networks in Sobolev spaces $W^{s,p}$
of functions between Euclidean vector spaces. For the convenience of the reader, we include the relevant implication for the present work. Recall that, throughout the paper, we assume that the activation function $\sigma$ is $C^\infty$, nonpolynomial and Lipschitz continuous.

\begin{lemma}
\label{lem:nn-univ}
Let $\Omega\subset \R^d$ be a bounded Lipschitz domain. Then for any function $u: \Omega \to \R^k$, where $u$ belongs to either the space of continuously differentiable functions $C^s(\bar{\Omega};\R^k)$, or the Sobolev space $W^{s,p}(\Omega;\R^k)$, for integer $s\ge 0$ and $p\in [1,\infty)$, and for any $\epsilon > 0$, there exists a neural network $\tilde{u}:\Omega \to \R^k$ with activation function $\sigma$, such that
\[
\sup_{x\in \Omega} |u(x) - \tilde{u}(x)| \le \epsilon.
\]
\end{lemma}
\begin{proof}
\textbf{Step 1:}
We first assume that $u \in C^{s}(\bar{\Omega};\R^k)$. By Lemma \ref{lem:periodic_extension}, there exists a (periodic) extension $U: \R^d\to \R$, such that $u(x) = U(x)$ for $x\in \Omega$, and $U\in C^s(\R^d;\R^k)$. It follows from \cite[Thm. 4.1]{pinkus1999approximation} that for any $\epsilon > 0$, there exists a neural network $\tilde{u}: \R^d\to \R$, such that
\[
\Vert u - \tilde{u} \Vert_{C^s(\bar{\Omega};\R^k)}
=
\sup_{|\alpha|\le s} \sup_{x\in \bar{\Omega}} |D^\alpha U(x) - D^\alpha\tilde{u}(x)| 
\le \epsilon.
\]

\textbf{Step 2:}
If $u\in W^{s,p}(\Omega;\R^k)$, then we note that by Lemma \ref{lem:moll}, the boundary-adapted mollification $u_\delta := \cM_\delta u$ converges to $u$ as $\delta \to 0$. Let $\epsilon > 0$ be given. Choose $\delta > 0$ sufficiently small, such that 
\[
\Vert u - u_\delta \Vert_{W^{s,p}(\Omega;\R^k)} \le \epsilon / 2.
\]
We note that $C^{s+1}(\bar{\Omega};\R^k)\embeds W^{s,p}(\Omega;\R^k)$ has a continuous embedding, and hence there exists a constant $C_0>0$, such that
\[
\Vert \slot \Vert_{W^{s,p}(\Omega;\R^k)} \le C_0\Vert \slot \Vert_{C^{s+1}(\bar{\Omega};\R^k)}.
\]
Now note that there exists a neural network $\tilde{u}$, such that $\Vert u_\delta - \tilde{u} \Vert_{C^{s+1}(\bar{\Omega};\R^k)}\le \epsilon /2C_0$; this follows
since $u_\delta \in C^{s+1}(\bar{\Omega};\R^k)$, by Step 1.
Combining these estimates, we obtain
\begin{align*}
\Vert u - \tilde{u} \Vert_{W^{s,p}(\Omega;\R^k)}
&\le
\Vert u - u_\delta \Vert_{W^{s,p}(\Omega;\R^k)}
+
\Vert u_\delta - \tilde{u} \Vert_{W^{s,p}(\Omega;\R^k)}
\\
&\le
\epsilon / 2 + C_0\Vert u_\delta - \tilde{u} \Vert_{C^{s+1}(\bar{\Omega};\R^k)}
\\
&\le \epsilon.
\end{align*}
This concludes our proof.
\end{proof}

\subsection{A Dense Subset Of Operators}
\label{app:dense}

One crucial ingredient in the proof of universal approximation for averaging neural operators is the fact that it is possible to reduce the problem for a general operator $\Psi^\dagger: C^s(\bar{\Omega};\R^k) \to C^{s'}(\bar{\Omega};\R^{k'})$ (or $\Psi^\dagger: W^{s,p}(\Omega;\R^{k}) \to W^{s',p'}(\Omega;\R^{k'})$, respectively), 
to one for a simpler class of operators which can be written in the form,
\begin{align}
\label{eq:simpler}
\tilde{\Psi}^\dagger(u) = \sum_{j=1}^J \alpha_j(u) \eta_j,
\end{align}
where $\eta_1,\dots, \eta_J$ are functions in $C^{s'}(\bar{\Omega};\R^{k'})$ (resp. in $W^{s',p'}(\Omega; \R^{k'})$), and  $\alpha_1,\dots, \alpha_J: L^1(\Omega;\R^k) \to \R$ are continuous nonlinear functionals, defined on the space of integrable functions. This is the content of the following two propositions. The first version is formulated for operators between spaces of continuously differentiable functions 
and the second between Sobolev spaces.

\begin{proposition}
\label{prop:simpler}
Let $\Psi^\dagger: C^s(\bar{\Omega}; \R^k) \to C^{s'}(\bar{\Omega};\R^{k'})$ be a continuous operator. Let $\mfK\subset C^s(\bar{\Omega};\R^k)$ be a compact subset. Then for any $\epsilon > 0$, there exist $J \in \N$, functions $\eta_1,\dots, \eta_J\in C^{s'}(\bar{\Omega};\R^{k'})$, and continuous functionals $\alpha_1,\dots, \alpha_J: L^1(\Omega;\R^k) \to \R$, such that the operator $\tilde{\Psi}^\dagger: L^1(\Omega;\R^k) \to C^{s'}(\bar{\Omega};\R^{k'})$, $\tilde{\Psi}^\dagger(u) := \sum_{j=1}^J \alpha_j(u) \eta_j$, satisfies
\[
\sup_{u\in \mfK} 
\left\Vert 
\Psi^\dagger(u) - \tilde{\Psi}^\dagger(u)
\right\Vert_{C^{s'}} \le \epsilon.
\]
\end{proposition}

\begin{remark}
We emphasize that even though the underlying operator $\Psi^\dagger(u)$ is only defined for $u\in C^s(\bar{\Omega};\R^k)$, and will generally not possess any continuous extension to an operator defined for $u\in L^1(\Omega;\R^k)$, the above Proposition \ref{prop:simpler} constructs a continuous operator $\tilde{\Psi}^\dagger: L^1(\Omega;\R^k) \to C^{s'}(\bar{\Omega};\R^{k'})$, 
whose restriction to compact $\mfK \subset C^s(\bar{\Omega};\R^k)$ provides a good approximation, $\Psi^\dagger|_{\mfK} \approx \tilde{\Psi}^\dagger|_\mfK$.
In this context note that $L^1(\Omega;\R^k) \supset C^s(\bar{\Omega};\R^k)$.
\end{remark}

\begin{proof}
Let us first point out that we may without loss of generality assume $k'=1$, in the following; indeed we can identify $C^{s'}(\bar{\Omega};\R^k) = [C^{s'}(\bar{\Omega};\R)]^k$, and thus it will suffice to approximate each component of the mapping $\Psi^\dagger: C^s(\bar{\Omega};\R^k) \to [C^{s'}(\bar{\Omega};\R)]^k$, individually. The $j$-th component defines a mapping $\Psi^\dagger_j: C^s(\bar{\Omega};\R^k) \to C^{s'}(\bar{\Omega};\R)$. Henceforth we use
the notation $C^{s'}(\bar{\Omega}):=C^{s'}(\bar{\Omega};\R).$

\textbf{Step 1: (construction of $\eta_1,\dots, \eta_J$) }
Our first goal is to construct suitable $\eta_1,\dots, \eta_J$. To this end, we consider $\mfK' := \Psi^\dagger(\mfK) \subset C^{s'}(\bar{\Omega})$. Note that since $\mfK$ is compact, and $\Psi^\dagger$ is continuous, its image $\mfK'$ is also compact. Let $B \supset \Omega$ be a bounding box, containing $\overline{\Omega}$ in its interior. By Lemma \ref{lem:periodic_extension}, there exists a continuous extension mapping
\[
\cE: C^{s'}(\bar{\Omega}) \to C^{s'}_{\mathrm{per}}(B),
\]
where $C^{s'}_\per(B)$ denotes the space of continuously differentiable functions $w: B \to \R$ on the Cartesian domain $B$, possessing periodic derivatives up to order $s'$. Since $\mfK' \subset C^{s'}(\bar{\Omega})$ is compact, it follows that also $\mfK'_\per := \cE(\mfK')$ is a compact subset of $C^{s'}_\per(B)$. Given this periodicity, we can identify $B \simeq \T^d$ with the periodic torus in a canonical way. Let $\eta_1,\eta_2,\dots$ denote an enumeration of the $L^2$-orthogonal (real) Fourier sine/cosine basis in $L^2_\per(B)$. Fix $\delta > 0$, and let $w_\delta: B \to \R$ be the standard mollification of the periodic function $w: B \to \R$. Since $\mfK'_\per \subset C^{s'}_\per(B)$ is compact, the set $\set{w_\delta}{w\in \mfK'_\per}$ is uniformly bounded in the $C^{r'}$-norm for any given $r'>s'$. In particular, it follows from classical results in harmonic analysis \cite[Chap. 1, Sect. 3]{jackson1930theory}, that approximation by Fourier series converges uniformly over $\set{w_\delta}{w\in \mfK'_\per}$, i.e.
\[
\lim_{J\to \infty}
\sup_{w\in \mfK'_\per} 
\left\Vert 
w_\delta - \textstyle\sum_{j=1}^J \langle w_\delta, \eta_j \rangle_{L^2} \eta_j
\right\Vert_{C^{s'}_\per(B)} = 0.
\]
Furthermore, $\lim_{\delta \to 0}\Vert w_\delta - w\Vert_{C^{s'}}=0$, uniformly over $\mfK'_\per$. In particular, given $\epsilon > 0$, we can first find $\delta = \delta(\epsilon) > 0$, such that
$
\sup_{w\in \mfK'_\per} \Vert w - w_\delta \Vert_{C^{s'}} \le \epsilon / 2,
$
and then $J \in \N$, such that 
$
\sup_{w\in \mfK'_\per} 
\left\Vert 
w_\delta - \textstyle\sum_{j=1}^J \langle w_\delta, \eta_j \rangle_{L^2} \eta_j
\right\Vert_{C^{s'}_\per(B)}
\le \epsilon/2.
$
It follows from the triangle inequality that 
\[
\sup_{w\in \mfK'_\per} 
\left\Vert 
w - \textstyle\sum_{j=1}^J \langle w_\delta, \eta_j \rangle_{L^2} \eta_j
\right\Vert_{C^{s'}_\per(B)}
\le \epsilon.
\]
This defines our choice of $\eta_1,\dots, \eta_J$. Note that we also have the identity
\[
\sum_{j=1}^J \langle w_\delta, \eta_j \rangle_{L^2} 
=
\sum_{j=1}^J \langle w, \eta_{j,\delta} \rangle_{L^2}.
\]
We now recall that $\Omega \subset B$, and $\cE$ is an extension operator, so that $\cE(v)|_\Omega = v$ for all $v\in \mfK'$. Furthermore, we have $\mfK'_\per = \cE(\mfK')$ and $\mfK' = \Psi^\dagger(\mfK)$, by definition. As a consequence, it follows that,
\begin{align*}
\sup_{u\in \mfK} 
\left\Vert 
\Psi^\dagger(u) - \textstyle\sum_{j=1}^J \langle \cE(\Psi^\dagger(u)), \eta_{j,\delta} \rangle_{L^2} \eta_j
\right\Vert_{C^{s'}(\bar{\Omega})}
&=
 \sup_{v\in \mfK'} 
\left\Vert 
v - \textstyle\sum_{j=1}^J \langle \cE(v), \eta_{j,\delta} \rangle_{L^2} \eta_j
\right\Vert_{C^{s'}(\bar{\Omega})}
\\
& =
 \sup_{v\in \mfK'} 
\left\Vert 
\cE(v) - \textstyle\sum_{j=1}^J \langle \cE(v), \eta_{j,\delta} \rangle_{L^2} \eta_j
\right\Vert_{C^{s'}(\bar{\Omega})}
\\
&\le
 \sup_{v\in \mfK'} 
\left\Vert 
\cE(v) - \textstyle\sum_{j=1}^J \langle \cE(v), \eta_{j,\delta} \rangle_{L^2} \eta_j
\right\Vert_{C^{s'}_\per(B)}
\\
&=
 \sup_{w\in \mfK'_\per} 
\left\Vert 
w - \textstyle\sum_{j=1}^J \langle w, \eta_{j,\delta} \rangle_{L^2} \eta_j
\right\Vert_{C^{s'}_\per(B)} \le \epsilon.
\end{align*}

\textbf{Step 2: (construction of $\alpha_1,\dots, \alpha_J$)}

Given the results of Step 1, let us define a nonlinear functional $\beta_j: C^{s}(\bar{\Omega};\R^k) \to \R$ by $\beta_j(u) := \langle \cE(\Psi^\dagger(u)), \eta_{j,\delta} \rangle_{L^2}$. Then, by Step 1, we have 
\begin{align}
\label{eq:beta}
\sup_{u\in \mfK} 
\left\Vert
\Psi^\dagger(u) - \textstyle\sum_{j=1}^J \beta_j(u) \eta_j
\right\Vert_{C^{s'}}
\le \epsilon.
\end{align}
This is almost the claimed result, except that $\beta_j$ does not define a continuous functional $L^1(\Omega;\R^k) \to \R$. To remedy this, we rely on mollification adapted to the bounded Lipschitz domain $\Omega$. Let $\delta > 0$ denote a mollification parameter. By Lemma \ref{lem:moll}, there exists a continuous operator $\cM_\delta: L^1(\Omega;\R^k) \to C^s(\bar{\Omega};\R^k)$, such that over the compact set $\mfK\subset C^s(\bar{\Omega};\R^k)$, we have
\[
\sup_{u\in \mfK} \Vert u - \cM_\delta u \Vert_{C^s} \to 0, 
\]
as $\delta \to 0$. 

We intend to define $\alpha_j: L^1(\Omega;\R^k) \to \R$ by $\alpha_j(u) := \beta_j(\cM_\delta u)$ for suitably chosen $\delta > 0$. Recall that for fixed $\delta_0>0$, the set $\mfK_\delta$ defined by 
\[
\mfK_\delta := \bigcup_{0\le \delta \le \delta_0} \cM_\delta(\mfK),
\]
is a compact subset of $C^s(\bar{\Omega};\R^k)$, by Lemma \ref{lem:Kdelta}. Note that for any $j=1,\dots, J$, we have a continuous mapping $\beta_j: \mfK_\delta \to \R$. Since $\mfK_\delta$ is compact, it follows that there exists a continuous modulus of continuity $\omega: [0,\infty) \to [0,\infty)$, satisfying $\omega(0) = 0$, such that 
\[
|\beta_j(u) - \beta_j(u')| \le \omega(\Vert u - u'\Vert_{C^s}), \quad \forall \, u,u'\in \mfK_\delta,
\]
holds for all $j=1,\dots, J$. It follows that 
\[
|\beta_j(u) - \beta_j(\cM_\delta u)|
\le
\omega(\Vert u - \cM_\delta u\Vert_{C^s}),
\]
for any $0<\delta \le \delta_0$. By Lemma \ref{lem:moll}, $\cM_\delta u \to u$ converges uniformly over $u\in \mfK$, so that we may conclude
\[
\lim_{\delta \to 0} \sup_{u\in \mfK} |\beta_j(u) - \beta_j(\cM_\delta u)| 
\le
\lim_{\delta \to 0} \omega\left(\textstyle\sup_{u\in \mfK} \Vert u - \cM_\delta u\Vert_{C^s}\right)
=
0.
\]
In particular, we can choose $\delta > 0$ sufficiently small, to ensure that $\alpha_j(u) := \beta_j(\cM_\delta u)$ satisfies
\begin{align}
\label{eq:alpha}
\sup_{u\in \mfK} |\beta_j(u) - \alpha_j(u)| \le \frac{\epsilon}{J \max_{j=1,\dots, J} \Vert \eta_j \Vert_{C^{s'}}},
\end{align}
for all $j=1,\dots, J$. Since $\delta > 0$, we also note that $\cM_\delta: L^1(\Omega;\R^k) \to C^s(\bar{\Omega};\R^k)$ is a continuous mapping, and hence $\alpha_j = \beta_j \circ \cM_\delta$ is continuous as a mapping $\alpha_j: L^1(\Omega;\R^k) \to \R$. 

\textbf{Step 3: (Conclusion)}
Combining \eqref{eq:beta} with \eqref{eq:alpha}, we obtain
\begin{align*}
\sup_{u\in \mfK} 
\left\Vert
\Psi^\dagger(u) - \textstyle\sum_{j=1}^J \alpha_j(u) \eta_j
\right\Vert_{C^{s'}}
&\le 
\sup_{u\in \mfK} 
\left\Vert
\Psi^\dagger(u) - \textstyle\sum_{j=1}^J \beta_j(u) \eta_j
\right\Vert_{C^{s'}}
\\
&\qquad 
+ \sup_{u\in \mfK} 
\left\Vert
\textstyle\sum_{j=1}^J [\beta_j(u)-\alpha_j(u)] \eta_j
\right\Vert_{C^{s'}}
\\
&\le 
\sup_{u\in \mfK} 
\left\Vert
\Psi^\dagger(u) - \textstyle\sum_{j=1}^J \beta_j(u) \eta_j
\right\Vert_{C^{s'}}
\\
&\qquad 
+ J
\max_{j=1,\dots, J}\Vert
\eta_j
\Vert_{C^{s'}} \max_{j=1,\dots, J} \sup_{u\in \mfK} 
|\beta_j(u)-\alpha_j(u)| 
\\
&\le 2\epsilon.
\end{align*}
As $\epsilon > 0$ was arbitrary, the claim of Proposition \ref{prop:simpler} now follows.
\end{proof}

We can also formulate a similar result for operators between Sobolev spaces. As the proof is almost identical to the proof of Proposition \ref{prop:simpler}, and essentially follows by replacing $C^s$ by $W^{s,p}$ and analogously in the
output space, we forego the details of this argument, 
and only state the final result.

\begin{proposition}
\label{prop:simpler2}
Let $\Psi^\dagger: W^{s,p}(\Omega; \R^k) \to W^{s',p'}(\Omega;\R^{k'})$ be a continuous operator. Let $\mfK\subset W^{s,p}(\Omega;\R^k)$ be a compact subset. Then for any $\epsilon > 0$, there exist $\eta_1,\dots, \eta_J\in W^{s',p'}(\Omega;\R^{k'})$, and continuous functionals $\alpha_1,\dots, \alpha_N: L^1(\Omega;\R^k) \to \R$, such that 
\[
\sup_{u\in \mfK} 
\left\Vert 
\Psi^\dagger(u) - \textstyle\sum_{j=1}^J \alpha_j(u) \eta_j
\right\Vert_{W^{s',p'}} \le \epsilon.
\]
\end{proposition}

\subsection{Approximation Of Nonlinear Functionals By The Averaging Neural Operator}
\label{app:alpha-approx}

Given the result of Proposition \ref{prop:simpler} and \ref{prop:simpler2} in the last section, a core ingredient in our proof of universal approximation for averaging neural operators will be the approximation of nonlinear functionals $\alpha: L^1(\Omega;\R^k)\to \R$ by averaging neural operators. In the present section, we will show that any nonlinear a functional $\alpha: L^1(\Omega;\R^k) \to \R$, $u \mapsto \alpha(u)$, can be approximated by an averaging neural operator in a suitable sense. This is the subject of the following lemma: 
\begin{lemma}
\label{lem:alpha-approx}
Let $\alpha: L^1(\Omega;\R^k) \to \R$ be a continuous nonlinear functional. Let $\mfK \subset L^1(\Omega;\R^k)$ be a compact set, consisting of bounded functions, $\sup_{u\in \mfK} \Vert u \Vert_{L^\infty}< \infty$. Then for any $\epsilon > 0$, there exists an averaging neural operator $\tilde{\alpha}: L^1(\Omega;\R^k) \to L^1(\Omega)$, all of whose output functions are constant so that we may also view $\tilde{\alpha}$
as a function
$\tilde{\alpha}: L^1(\Omega;\R^k) \to \R$, such that 
\[
\sup_{u\in \mfK} |\alpha(u) - \tilde{\alpha}(u)| \le \epsilon.
\]
\end{lemma}

\begin{remark}
\label{rem:alpha-approx}
We note that the construction in our proof of Lemma \ref{lem:alpha-approx} (cp. \eqref{eq:psi1} below) actually defines an averaging neural operator with (1) a lifting layer represented by $R(u(x),x)$, (2) a \emph{single} hidden layer $v(x) \mapsto \sigma\left(\fint_{\Omega} v(y) \, dy\right)$, and (3) the projection layer $Q(v,x) := q_1(v)$, where $q_1$ is a neural network depending only on $v$. Hence, the result of Lemma \ref{lem:alpha-approx}, and as a consequence the universal approximation property as described in Theorems \ref{thm:universal10} and \ref{thm:universal20}, can be achieved with a \emph{single} evaluation of the averaging operation: $L=1$.
\end{remark}

\begin{proof}{(Lemma \ref{lem:alpha-approx})}
Let $\alpha: L^1(\Omega;\R^k) \to \R$ be a continuous functional. Let $\mfK\subset L^1(\Omega;\R^k)$ be a compact set. Fix $\epsilon > 0$. Our aim is to show that there exists an averaging neural operator $\tilde{\alpha}: L^1(\Omega;\R^k) \to L^1(\Omega)$ with constant output, such that 
\[
\sup_{u\in \mfK} | \alpha(u) - \tilde{\alpha}(u) | \le \epsilon.
\]

To prove this, note that we can identify any $u\in L^1(\Omega;\R^k)$ with a function in $L^1(B;\R^k)$, via an extension of $u(x) := 0$ for $x\in B\setminus \Omega$. Using this identification, the compact subset $\mfK\subset L^1(\Omega;\R^k)$ can be identified with a compact subset of $L^1(\R^d;\R^k)$. Fix a smooth mollifier $\rho\in C^\infty$, and denote $\rho_\delta(x) := \delta^{-d} \rho(x/\delta)$ for $\delta > 0$. We denote by $u_\delta(x) = (u\ast \rho_\delta)(x)$ the mollification of $u$ (extended to all of $\R^d$ by $0$ outside of $\Omega$). Since $\mfK\subset L^1(\Omega;\R^k)$ is compact, it follows that 
\begin{align*}
\lim_{\delta \to 0} \sup_{u\in \mfK} \Vert u - u_\delta \Vert_{L^1(\Omega;\R^k)} = 0.
\end{align*}
Since $\alpha: L^1(\Omega;\R^k) \to \R$ is continuous, it follows that the mapping $\alpha_\delta: L^1(B,\R^k) \to \R$ defined by $\alpha_\delta(u) := \alpha(u_\delta)$, for $\delta>0$, converges uniformly over $\mfK$, as $\delta\to 0$:
\begin{align}
\label{eq:adelta}
\lim_{\delta \to 0} \sup_{u\in \mfK} |\alpha(u) - \alpha_\delta(u)|
\le
\lim_{\delta \to 0} \sup_{u\in \mfK} |\alpha(u) - \alpha(u_\delta)|
= 0.
\end{align}

The idea is now that, for any choice of orthonormal basis $\xi_1,\xi_2,\dots$ of $L^2(\Omega;\R^k)$, which we may additionally choose to be smooth, $\xi_j \in C^\infty(\Omega;\R^k)$, we have uniform convergence,
\begin{align}
\label{eq:phiconv}
\sup_{u\in \mfK} 
\left\Vert 
u_\delta - 
\textstyle\sum_{j=1}^J \langle u_\delta, \xi_j \rangle_{L^2} \xi_j
\right\Vert_{L^1(B;\R^k)} 
\to 0,
\end{align}
as $J\to \infty$. We note that the expression in \eqref{eq:phiconv} is well-defined, since $\mfK \subset L^1\cap L^\infty \subset L^2$, by assumption. Furthermore, $\mfK\subset L^2$ is compact in the $L^2$-norm; this entails the uniform convergence \eqref{eq:phiconv}. With the convention that $u$ and $\xi_j$ are expanded by $0$ outside of $\Omega$, we have
\begin{gather}
\label{eq:L2}
\begin{aligned}
\langle u_\delta, \xi_j \rangle_{L^2}
&=
\int_{\Omega} u_\delta(y) \cdot \xi_j(y) \, dy
= \int_{\R^d} (u \ast \rho_\delta)(y) \cdot \xi_j(y) \, dy
\\
&= \int_{\R^d} u(y) \cdot (\xi_j \ast \rho_\delta)(y) \, dy
= \fint_{\Omega} u(y) \cdot {\xi}_{j,\delta}(y) \, dy,
\end{aligned}
\end{gather}
where we have defined ${\xi}_{j,\delta}(y) := |\Omega| (\xi_j \ast \rho_\delta)(y)$.
Hence, if we define $\alpha_{\delta,J}: L^1(\Omega;\R^k) \to \R$ by 
\[
\alpha_{\delta,J}(u) :=
\alpha\left(
\sum_{j=1}^J \left(\fint_{\Omega} u(y) \cdot \xi_{j,\delta}(y) \, dy \right) \, \xi_j
\right),
\]
then it follows from \eqref{eq:adelta}, \eqref{eq:phiconv} and \eqref{eq:L2}, that for $\delta = \delta(\mfK)>0$ sufficiently small, and $J = J(\delta, \mfK)\in \N$ sufficiently large, we have $\sup_{u\in \mfK} |\alpha_\delta(u) - \alpha_{\delta,J}(u)| \le \epsilon/2$. To indicate the connection to averaging neural operators, we note that we can write $\alpha_{\delta, J}$ as the following composition:
\[
u 
\mapsto 
\left( \fint_{\Omega} u(y) \cdot \xi_{1,\delta}(y) \, dy, \dots, \fint_{\Omega} u(y) \cdot\xi_{J,\delta}(y) \, dy \right)
\mapsto 
\alpha\left(
\sum_{j=1}^J \left(\fint_{\Omega} u(y) \cdot\xi_{j,\delta}(y) \, dy \right) \, \xi_j
\right).
\]
The first of these mapping requires computation of certain averages (albeit against a function $\xi_{j,\delta}$). Approximation of this mapping clearly requires nonlocality. The second mapping defines a continuous function $\beta: \R^J \to \R$, $c = (c_1,\dots, c_J) \mapsto \alpha\left(\sum_{j=1}^J c_j \xi_j\right)$. Approximation of this mapping is possible by ordinary neural networks; this requires nonlinearity, but does not require nonlocality. 

Choose $M>0$, such that the image of the compact set $\mfK\subset L^1(\Omega;\R^k)$ under the mapping
\[
L^1(\Omega;\R^k) \to \R^J, 
\quad
u \mapsto \left( \fint_{\Omega} u(y) \cdot \xi_{1,\delta}(y) \, dy, \dots, \fint_{\Omega} u(y) \cdot\xi_{J,\delta}(y) \, dy \right),
\]
is contained in $[-M,M]^J$. Note that $c \mapsto \beta(c) = \alpha\left(\sum_{j=1}^J c_j \xi_j\right)$ is a continuous function, by the continuity of $\alpha$. By universality of conventional neural networks, there exists a neural network $\tilde{\beta}: \R^J \to \R$, such that 
\begin{align}
\label{eq:betaapprox}
\sup_{c\in [-M,M]^N} |\tilde{\beta}(c) - \beta(c)| \le \epsilon.
\end{align}
Write $\tilde{\beta}: \R^J \to \R$ as a composition of its hidden layers:
\begin{gather}
\label{eq:tildepsi}
\left\{
\begin{aligned}
v_0 &:= c, 
\\
v_{\ell} &:= \sigma\left(\tilde{A}_\ell v_{\ell-1} + \tilde{b}_\ell \right), \quad \ell =1,\dots, L, 
\\
\tilde{\beta}(c) &:= \tilde{A}_{L+1} v_L + \tilde{b}_{L+1},
\end{aligned}
\right.
\end{gather}
where $\tilde{A}_\ell$, $\tilde{b}_\ell$ are the weights and biases of the hidden layers. Parallelizing the neural networks constructed in Lemma \ref{lem:proj}, below, it follows that for any $\epsilon'>0$, there exists a neural network $\tilde{R}: \R^k \times \Omega \to \R^J$, $(v,x) \mapsto \tilde{R}(v,x) = (\tilde{R}_1(v,x),\dots, \tilde{R}_J(v,x))$, such that 
\[
\sup_{u\in \mfK}
\left|
\fint_{\Omega} \tilde{R}_j(u(y),y) \, dy
-
\fint_{\Omega} u(y)\cdot \xi_{j,\delta}(y) \, dy
\right| 
\le \epsilon',
\]
for $j=1,\dots, J$. Composing the output layer with an affine mapping, we can construct
another neural network ${R}$, such that
\[
R(u(x),x) = \tilde{A}_1 \tilde{R}(u(x),x) + \tilde{b}_1,
\]
where $\tilde{A}_1$, $\tilde{b}_1$ are the weights and biases of the input layer of $\tilde{\beta}$ defined by \eqref{eq:tildepsi}.
In particular, it follows that for any input $u\in \mfK$, and defining coefficients $c(u) := (c_1,\dots, c_J)$ by $c_j := \fint_{\Omega} u(y)\cdot \xi_{j,\delta}(y) \, dy$, we have
\begin{align}
\label{eq:Rc}
\sup_{u\in \mfK}
\left|
\sigma \left(
\fint_{\Omega} {R}(u(y),y) \, dy
\right)
-
\sigma\left(\tilde{A}_1 c(u) + \tilde{b}_1\right)
\right| 
\le \Vert \sigma \Vert_{\Lip} \Vert \tilde{A}_1 \Vert \epsilon',
\end{align}
where $\Vert \tilde{A}_1 \Vert$ denotes the operator norm of $\tilde{A}_1$. Note that the second term in \eqref{eq:Rc} is exactly the output of the first hidden layer of $\tilde{\beta}$, cp. \eqref{eq:tildepsi}. Let us decompose $\tilde{\beta}(c) = \tilde{\beta}_{1} \circ \sigma(\tilde{A}_1 c + \tilde{b}_1)$, where $\tilde{\beta}_1$ denotes the composition of the other hidden layers, $\ell=2,\dots, L$, and the output layer. Composing each of the terms appearing on the left-hand side of \eqref{eq:Rc} with $\tilde{\beta}_1$ and choosing $\epsilon'$ sufficiently small (depending on $\tilde{\beta}_1$, $\tilde{A}_1$), we can ensure that
\[
\sup_{u\in \mfK}
\left|
\tilde{\beta}_1 \left( \sigma\left( 
\fint_{\Omega} {R}(u(y),y) \, dy
\right) \right)
-
\tilde{\beta}_1\left(\sigma\left(\tilde{A}_1 c(u) + \tilde{b}_1\right)\right)
\right|
\le \epsilon,
\]
or equivalently,
\begin{align}
\label{eq:tpsiapprox}
\sup_{u\in \mfK}
\left|
\tilde{\beta}_1 \left( \sigma\left( 
\fint_{\Omega} {R}(u(y),y) \, dy
\right) \right)
-
\tilde{\beta}(c(u))
\right|
\le \epsilon.
\end{align}
From the above, \eqref{eq:betaapprox} an \eqref{eq:tpsiapprox}, it follows that the averaging neural operator $\tilde{\alpha}$, defined by the following composition,
\begin{align}
\label{eq:psi1}
\tilde{\alpha}: u 
\overset{(1)}{\mapsto} R(u(\slot),\slot) 
\overset{(2)}{\mapsto}  \sigma\left(\fint_{\Omega} R(u(y),y) \, dy \right)
\overset{(3)}{\mapsto}  \tilde{\beta}_1 \circ \sigma\left(\fint_{\Omega} R(u(y),y) \, dy \right),
\end{align}
satisfies
\begin{align*}
\sup_{u\in \mfK} |\alpha(u) - \tilde{\alpha}(u)|
&=
\sup_{u\in \mfK} |\beta(c(u)) - \tilde{\alpha}(u)|
\\
&\le
\sup_{u\in \mfK} |\beta(c(u)) - \tilde{\beta}(c(u))|
\\
&\qquad +
\sup_{u\in \mfK} \left|\tilde{\beta}(c(u)) - \tilde{\beta}_1\circ \sigma\left(\fint_{\Omega} R(u(y),y) \, dy \right)\right|
\\
&\le
\sup_{c\in [-M,M]^N} |\beta(c) - \tilde{\beta}(c)|
\\
&\qquad +
\sup_{u\in \mfK} \left|\tilde{\beta}(c(u)) - \tilde{\beta}_1\circ \sigma\left(\fint_{\Omega} R(u(y),y) \, dy \right)\right|
\\
&\le \epsilon.
\end{align*}
\end{proof}

We finally state the following lemma, which was used in the preceding
proof of Lemma \ref{lem:alpha-approx}:

\begin{lemma}
\label{lem:proj}
Let $\mfK\subset L^1(\Omega;\R^k)$ be compact, consisting of uniformly bounded functions $\sup_{u\in \mfK} \Vert u \Vert_{L^\infty} < \infty$. Let ${\xi} \in C^\infty(\bar{\Omega};\R^k)$ be given and fixed. Then for any $\epsilon > 0$, there exists a neural network $\tilde R: \R^k \times \Omega \to \R^k$, such that
\[
\sup_{u\in \mfK} \left| \fint_{\Omega} \tilde R(u(y),y) \, dy - \fint_{\Omega} u(y)\cdot {\xi}(y) \, dy \right|
\le
\epsilon.
\]
\end{lemma}

\begin{proof}
Fix $\epsilon > 0$. In the following, we denote by 
\[
M_\mfK := \sup_{u\in \mfK} \Vert u \Vert_{L^\infty},
\]
the upper $L^\infty$-bound on elements $u\in \mfK$. By assumption, $M_\mfK$ is finite. As $\xi \in C^\infty(\bar{\Omega};\R^k)$ is fixed, and $\Omega\subset \R^d$ is bounded, there exists a neural network $\tilde{\xi}: \Omega \to \R^k$, 
such that 
\[
\sup_{x\in \Omega} \vert \xi(x) - \tilde{\xi}(x) \vert_{\ell^\infty} \le \epsilon/(2kM_\mfK),
\]
where $k$ is the number of components of $\xi$. Furthermore define
\[
M_\xi := 
\max\left\{
\Vert \xi \Vert_{L^\infty}, \Vert \tilde\xi \Vert_{L^\infty}
\right\} < \infty.
\]
By the universality of ordinary neural networks, there exists a neural network $\tilde{\times}: [-M_\mfK,M_\mfK]^k \times [-M_\xi,M_\xi]^k \to \R$, such that 
\[
\sup_{\vert v \vert_{\ell^\infty} \le M_\mfK, \vert w\vert_{\ell^\infty} \le M_\xi}
| v \cdot w  - \tilde{\times}(v,w) | \le \epsilon/2.
\]
Defining a new neural network as the composition $\tilde R(v,x) := \tilde{\times}(v,\tilde{\xi}(x))$, it now follows that for any $u\in \mfK$:
\begin{align*}
\sup_{x\in \Omega} |\tilde R(u(x),x) - u(x)\cdot \xi(x)|
&=
\sup_{x\in \Omega} |\tilde{\times}(u(x),\tilde\xi(x)) - u(x)\cdot \xi(x)|
\\
&\le 
\sup_{x\in \Omega} |\tilde{\times}(u(x),\tilde\xi(x)) - u(x)\cdot \tilde{\xi}(x)|
\\
&\qquad + \sup_{x\in \Omega} | u(x)\cdot \tilde{\xi}(x) - u(x) \cdot \xi(x)|
\\
&\le 
\sup_{\vert v \vert_{\ell^\infty} \le M_\mfK, \vert w\vert_{\ell^\infty} \le M_\xi} |\tilde{\times}(v,w) - v\cdot w|
\\
&\qquad + kM \sup_{x\in \Omega} \vert \tilde{\xi}(x) - \xi(x) \vert_{\ell^\infty}
\\
&\le \epsilon.
\end{align*}
In particular, this upper bound implies that 
\[
\sup_{u\in \mfK}
\left| 
\fint_{\Omega} \tilde R(u(y),y) \, dy - \fint_{\Omega} u(y) \cdot \xi(y) \, dy
\right|
\le \epsilon.
\]
\end{proof}

\subsection{Proof Of Universal Approximation $C^s \to C^{s'}$, 
Theorem \ref{thm:universal10}}
\label{app:universal1}

Given the results of Sections \ref{app:dense} and \ref{app:alpha-approx}, we can now provide a detailed proof of the universal approximation Theorem \ref{thm:universal10}. The main idea behind the proof is that using Proposition \ref{prop:simpler} in Section \ref{app:dense}, it suffices to approximate operators of the form 
\[
\tilde{\Psi}^\dagger(u) = \sum_{j=1}^J \alpha_j(u) \eta_j,
\]
where $\alpha_j: L^1(\Omega;\R^k) \to \R$ are functionals, and $\eta_j\in C^{s'}(\bar{\Omega};\R^{k'})$ are fixed functions. Given averaging neural operators $\Psi_1,\dots, \Psi_J$, we can easily combine them to obtain a new averaging neural operator $\Psi$, such that $\Psi(u) = \sum_{j=1}^J \Psi_j(u)$ for all input functions $u$. Thus, it suffices to consider only one term in the above sum, i.e. it will suffice to prove that any mapping of the form 
\[
\tilde{\Psi}^\dagger_j: 
L^1(\Omega;\R^k) \to C^{s'}(\bar{\Omega};\R^{k'}), \quad u \mapsto \alpha_j(u) \eta_j,
\]
can be approximated by an averaging neural operator. The main difficulty in proving this is to show that any functional $\alpha_j: L^1(\Omega;\R^k) \to \R$ can be approximated by averaging neural operators. This is the content of Lemma \ref{lem:alpha-approx} in the previous Section \ref{app:alpha-approx}. Given this core ingredient, we can now provide a complete the proof of the universal approximation Theorem \ref{thm:universal10}, below:

\begin{proof}{(Theorem \ref{thm:universal10})}
\label{pf:universal1}

Let $\Psi^\dagger: C^s(\bar{\Omega};\R^k) \to C^{s'}(\bar{\Omega};\R^{k'})$ be a continuous operator. Let $\mfK\subset C^s(\bar{\Omega};\R^k)$ be compact. We aim to show that for any $\epsilon > 0$, there exists an averaging neural operator $\Psi$ of the form $\Psi = \cQ \circ \cL_L \circ \dots \circ \cL_1 \circ \cR$, such that 
\[
\sup_{u\in \mfK} \Vert \Psi^\dagger(u) - \Psi(u) \Vert_{C^{s'}} \le \epsilon.
\]
Fix $\epsilon > 0$. By Proposition \ref{prop:simpler}, there exist functions $\eta_1,\dots, \eta_J \in C^{s'}(\bar{\Omega}; \R^{k'})$, and continuous functionals $\alpha_1,\dots, \alpha_J: L^1(\Omega;\R^k) \to \R$, such that 
\[
\sup_{u\in \mfK} 
\left\Vert 
\Psi^\dagger(u) - \textstyle\sum_{j=1}^J \alpha_j(u) \eta_j
\right\Vert_{C^{s'}} \le \epsilon/2.
\]
We now make the following claim:
\begin{claim}
\label{claim:cite}
Let $\mfK\subset L^1(\Omega;\R^k)$ be a compact set, consisting of bounded functions $\sup_{u\in \mfK} \Vert u \Vert_{L^\infty} < \infty$.
Let $\eta_1,\dots, \eta_J\in C^{s'}(\bar{\Omega};\R^k)$ be functions and let $\alpha_1,\dots,\alpha_J: L^1(\Omega;\R^k) \to \R$ be continuous nonlinear functionals. Then for any $j=1,\dots, J$, there exists an averaging neural operator $\Psi_j: L^1(\Omega;\R^k) \to C^{s'}(\bar{\Omega};\R^k)$, such that 
\[
\sup_{u\in \mfK} \Vert \alpha_j(u) \eta_j - \Psi_j(u) \Vert_{C^{s'}} \le \epsilon/2J.
\]
\end{claim}
Relying on the above claim, it is easy to see that there exists an averaging neural operator, such that $\Psi(u) = \sum_{j=1}^J \Psi_j(u)$ for all input functions $u$. This operator satisfies, for any $u\in \mfK$:
\begin{align*}
\Vert \Psi^\dagger(u) - \Psi(u) \Vert_{C^{s'}}
&\le 
\left\Vert \Psi^\dagger(u) - \textstyle\sum_{j=1}^J \alpha_j(u) \eta_j \right\Vert_{C^{s'}} + 
\left\Vert \textstyle\sum_{j=1}^J \alpha_j(u) \eta_j - \Psi(u) \right\Vert_{C^{s'}}
\\
&\le
\left\Vert \Psi^\dagger(u) - \textstyle\sum_{j=1}^J \alpha_j(u) \eta_j \right\Vert_{C^{s'}} + 
 \sum_{j=1}^J \left\Vert\left[\alpha_j(u) \eta_j - \Psi_j(u)\right] \right\Vert_{C^{s'}}
\\
&\le \epsilon/2 + J \epsilon / 2J = \epsilon,
\end{align*}
where the last estimate holds for any $u\in \mfK$. Thus, taking the supremum over $\mfK$, the above claim implies that
\[
\sup_{u\in \mfK} \Vert \Psi^\dagger(u) - \Psi(u) \Vert_{C^{s'}}
\le \epsilon,
\]
which gives the universal approximation property of Theorem \ref{thm:universal10}.
To finish our argument, it thus remains to prove the above Claim \ref{claim:cite}.

To prove the claim, fix $j\in \{1,\dots, J\}$ for all the following. We first define
\begin{align}
\label{eq:Mphi}
M_\eta := \Vert \eta_j \Vert_{C^{s'}(\Omega;\R^{k'})}.
\end{align}
Next, we observe that by Lemma \ref{lem:alpha-approx}, there exists an averaging neural operator $\tilde{\alpha}_j: L^1(\Omega;\R^k) \to L^1(\Omega)$, with constant output functions, such that
\begin{align}
\label{eq:aest}
\sup_{u\in \mfK} |\alpha_j(u) - \tilde{\alpha}_j(u)| \le \epsilon / (6J M_\eta).
\end{align}
Choose $M_\alpha > 0$, such that 
\begin{align}
\label{eq:Malpha}
\sup_{u\in \mfK}|\alpha_j(u)|, \; \sup_{u\in \mfK} |\tilde{\alpha}_j(u)| \le M_\alpha.
\end{align}
Since $\eta_j \in C^{s'}(\bar{\Omega};\R^{k'})$, there exists an ordinary neural network 
$\tilde{\eta}_j: \Omega \to \R^{k'}$, such that (cp. e.g. the universal approximation result of Lemma \ref{lem:nn-univ}):
\begin{align}
\label{eq:Phiest}
\Vert \eta_j - \tilde{\eta}_j \Vert_{C^{s'}} \le \epsilon /(6JM_\alpha).
\end{align}
Let us also define
\begin{align}
\label{eq:MPhi}
\tilde{M}_\eta := \max\left\{ M_\eta, \Vert \tilde{\eta}_j \Vert_{C^{s'}}\right\}.
\end{align}
Fix a small parameter $\delta > 0$, to be determined below. Since scalar multiplication $\times: [-M_\alpha,M_\alpha]\times [-\tilde{M}_\eta,\tilde{M}_\eta]^k \to \R^k$, $(a, v) \mapsto \times(a,v) := a v$ defines a smooth mapping, there similarly exists a neural network $\tilde{\times}: [-M_\alpha,M_\alpha]\times [-\tilde{M}_\eta,\tilde{M}_\eta]^k \to \R^k$, $(a, v) \mapsto \times(a,v)$, such that
\begin{align}
\label{eq:timesest}
\Vert \times(\slot, \slot) - \tilde{\times}(\slot,\slot) \Vert_{C^{s'}([-M_\alpha,M_\alpha]\times [-\tilde{M}_\eta,\tilde{M}_\eta]^{k'};\R^{k'})} \le \delta.
\end{align}
We also recall that a composition of $C^{s'}$-functions is itself $C^{s'}$, and that there exists a constant $C_0 = C_0(s',k')>0$, depending only on $s'$ and $k'$, such that
\begin{align*}
&\sup_{|a|\le M_\alpha} \Vert \times(a,\tilde{\eta}_j(\slot)) - \tilde{\times}(a, \tilde{\eta}_j(\slot)) \Vert_{C^{s'}(\Omega;\R^{k'})}
\\
&\qquad \le
\Vert \times(\slot,\tilde{\eta}_j(\slot)) - \tilde{\times}(\slot, \tilde{\eta}_j(\slot)) \Vert_{C^{s'}([-M_\alpha,M_\alpha]\times\Omega;\R^{k'})}
\\
&\qquad \le
C_0\Vert \times(\slot,\tilde{\eta}_j(\slot)) - \tilde{\times}(\slot,\slot) \Vert_{C^{s'}([-M_\alpha,M_\alpha]\times [-\tilde{M}_\eta,\tilde{M}_\eta]^{k'};\R^{k'})}
\; \Vert \tilde{\eta}_j \Vert_{C^{s'}(\Omega;\R^{k'})}
\\
&\qquad \le
C_0\delta \tilde{M}_\eta.
\end{align*}
Thus, for any $|a|, |\tilde{a}| \le M_\alpha$, we obtain
\begin{align*}
&\Vert \times(a,\eta_j(\slot)) - \tilde{\times}(\tilde{a}, \tilde{\eta}_j(\slot)) \Vert_{C^{s'}(\Omega;\R^{k'})}
\\
&\hspace{2cm} \le
\Vert \times(a,\eta_j(\slot)) -  \times(\tilde{a},\eta_j(\slot)) \Vert_{C^{s'}(\Omega;\R^{k'})} 
\\
&\hspace{2cm}\quad  
+ \Vert \times(\tilde{a},\eta_j(\slot)) -  \times(\tilde{a},\tilde{\eta}_j(\slot)) \Vert_{C^{s'}(\Omega;\R^{k'})} 
\\
&\hspace{2cm}\quad  
+\Vert \times(\tilde{a},\tilde{\eta}_j(\slot)) - \tilde{\times}(\tilde{a}, \tilde{\eta}_j(\slot)) \Vert_{C^{s'}(\Omega;\R^{k'})}
\\
&\hspace{2cm} \le
|a-\tilde{a}| M_\eta 
+ M_\alpha \Vert \eta_j - \tilde{\eta}_j \Vert_{C^{s'}(\Omega;\R^{k'})} 
+ C_0 \delta \tilde{M}_\eta.
\end{align*}
Recalling \eqref{eq:aest} and \eqref{eq:Phiest} to bound the first two terms, and choosing $\delta := \epsilon / (6N C_0 \tilde{M}_\eta)$ in \eqref{eq:timesest}, it follows that
\[
\Vert \times(a,\eta_j(\slot)) - \tilde{\times}(\tilde{a}, \tilde{\eta}_j(\slot)) \Vert_{C^{s'}(\Omega;\R^{k'})}
\le 
\epsilon/6J
+ \epsilon/ 6J
+ \epsilon / 6J,
\]
for any $|a|$, $|\tilde{a}|\le M_\alpha$. By definition of $M_\alpha$, \eqref{eq:Malpha}, given arbitrary $u\in \mfK$, we have $|\alpha_j(u)| \le M_\alpha$ and $|\tilde{\alpha}_j(u)|\le M_\alpha$. Hence, the above estimate finally implies that 
\begin{align}
\label{eq:almost}
\sup_{u\in \mfK}
\Vert \alpha_j(u) \eta_j - \tilde{\times}(\tilde{\alpha}_j(u), \tilde{\eta}_j(\slot)) \Vert_{C^{s'}(\Omega;\R^{k'})} \le \epsilon / 2J.
\end{align}
Note that the mapping $\R \times \Omega \to \R^{k'}$, $(a,x) \mapsto \tilde{\times}(a,\tilde{\eta}_j(x))$ is an ordinary neural network in $(a,x)$. Since $\tilde{\alpha}_j$ is an averaging neural operator by construction, we can write it in the form $\tQ \circ \cL_L\circ \dots \circ \cL_1 \circ \cR$, in terms of a raising operator $\cR$, hidden layers $\cL_\ell$, and a projection layer $\tQ$, where the values $\tQ(v)(x) := \tilde{Q}(v(x),x) \in \R$ are given in terms of an ordinary neural network $\tilde{Q}$. Let $\cQ$ denote the composition $\cQ(v)(x) := \tilde{\times}(\tilde{Q}(v(x),x),\tilde{\eta}_j(x))$. Then 
\[
\Psi(u) := \cQ \circ \cL_L \circ \dots \circ \cL_1 \circ \cR(u),
\]
defines an averaging neural operator, for which 
\[
\Psi(u)(x) = \tilde{\times}(\tilde{\alpha}_j(u),\tilde{\eta}_j(x)).
\]
By \eqref{eq:almost}, it follows that
\[
\Vert \alpha_j(u) \eta_j - \Psi(u) \Vert_{C^{s'}}
\le 
\epsilon / 2J.
\]
This concludes our proof of the claim.
\end{proof}

\subsection{Proof Of Universal Approximation $W^{s,p} \to W^{s',p'}$, Theorem \ref{thm:universal20}}
\label{app:universal2}

The previous section provides the detailed proof of the universal approximation theorem \ref{thm:universal10}, in the setting where the underlying function spaces $\cX$ and $\cY$ consist of continuously differentiable functions. Theorem \ref{thm:universal20} states a corresponding universality result in the scale of Sobolev spaces. The proof of Theorem \ref{thm:universal20} is almost identical to the proof of Theorem \ref{thm:universal10}. In the present section, we provide the necessary alterations to the proof.

\begin{proof}{(Theorem \ref{thm:universal20})}
Let $\Psi^\dagger: W^{s,p}(\Omega;\R^k) \to W^{s',p'}(\Omega;\R^{k'})$ be a continuous operator with integer $s,s'\ge 0$ and $p,p'\in [1,\infty)$. Let $\mfK\subset W^{s,p}(\Omega;\R^k)$ be compact, consisting of bounded functions, $\sup_{u\in \mfK} \Vert u \Vert_{L^\infty} < \infty$. We aim to show that for any $\epsilon > 0$, there exists an averaging neural operator $\Psi$ of the form $\Psi = \cQ \circ \cL_L \circ \dots \circ \cL_1 \circ \cR$, such that 
\[
\sup_{u\in \mfK} \Vert \Psi^\dagger(u) - \Psi(u) \Vert_{W^{s',p'}} \le \epsilon.
\]
Fix $\epsilon > 0$. By Proposition \ref{prop:simpler2}, there exist functions $\eta_1,\dots, \eta_J \in W^{s',p'}(\Omega; \R^{k'})$, and continuous functionals $\alpha_1,\dots, \alpha_J: L^1(\Omega;\R^k) \to \R$, such that 
\[
\sup_{u\in \mfK} 
\left\Vert 
\Psi^\dagger(u) - \textstyle\sum_{j=1}^J \alpha_j(u) \eta_j
\right\Vert_{W^{s',p'}} \le \epsilon/3.
\]
Approximating each $\eta_j$ by its boundary-adapted mollification, $\eta_{j,\delta} := \cM_\delta \eta_j \in C^\infty(\bar{\Omega};\R^{k'})$ with $\delta > 0$ chosen sufficiently small (cp. Lemma \ref{lem:moll}), we can ensure that 
\[
\sup_{u\in \mfK} 
\left\Vert 
\Psi^\dagger(u) - \textstyle\sum_{j=1}^J \alpha_j(u) \eta_{j,\delta}
\right\Vert_{W^{s',p'}} \le 2\epsilon/3.
\]
We next note that $C^{s'+1}(\bar{\Omega};\R^k) \embeds W^{s',p'}(\Omega;\R^{k'})$ has a continuous embedding, hence there exists a constant $C_0>0$, such that 
\begin{align}
\label{eq:embedC}
\Vert v \Vert_{W^{s',p'}} \le C_0 \Vert v \Vert_{C^{s'+1}}, \quad \forall\, v\in C^{s'+1}(\bar{\Omega};\R^{k'}).
\end{align}
We note that $u \mapsto \textstyle\sum_{j=1}^J \alpha_j(u) \eta_{j,\delta}$ defines a continuous operator $L^1(\Omega;\R^k) \to C^{s'+1}(\bar{\Omega};\R^{k'})$ and recall that, by assumption, $\mfK\subset  W^{s,p}(\Omega;\R^k) \subset L^1(\Omega;\R^k)$ is a compact set consisting of bounded functions, $\sup_{u\in \mfK} \Vert u \Vert_{L^\infty}< \infty$. It thus follows from Claim \ref{claim:cite} 
that there exists an averaging neural operator $\Psi: L^1(\Omega;\R^k) \to C^{s'+1}(\bar{\Omega};\R^{k'})$, such that
\[
\sup_{u\in \mfK} 
\left\Vert 
\textstyle\sum_{j=1}^J \alpha_j(u) \eta_{j,\delta}
- \Psi(u)
\right\Vert_{C^{s'+1}}
\le \epsilon/3C_0,
\]
where $C_0>0$ denotes the embedding constant of \eqref{eq:embedC}. For this averaging neural operator $\Psi$, it follows that
\begin{align*}
\sup_{u\in \mfK} 
\left\Vert 
\Psi^\dagger(u) - \Psi(u)
\right\Vert_{W^{s',p'}}
&\le
\sup_{u\in \mfK} 
\left\Vert 
\Psi^\dagger(u) - \textstyle\sum_{j=1}^J \alpha_j(u) \eta_{j,\delta}
\right\Vert_{W^{s',p'}}
\\
&\qquad 
+ \sup_{u\in \mfK} 
\left\Vert 
\textstyle\sum_{j=1}^J \alpha_j(u) \eta_{j,\delta} -\Psi(u)
\right\Vert_{W^{s',p'}}
\\
&\le
\sup_{u\in \mfK} 
\left\Vert 
\Psi^\dagger(u) - \textstyle\sum_{j=1}^J \alpha_j(u) \eta_{j,\delta}
\right\Vert_{W^{s',p'}}
\\
&\qquad 
+ C_0\sup_{u\in \mfK} 
\left\Vert 
\textstyle\sum_{j=1}^J \alpha_j(u) \eta_{j,\delta} -\Psi(u)
\right\Vert_{C^{s'+1}}
\\
&\le \epsilon.
\end{align*}
This concludes our proof.
\end{proof}


\end{document}